\documentclass[reqno,12pt]{amsart}

\usepackage[T1]{fontenc}
\usepackage[expansion=false]{microtype}
\usepackage{amsfonts,amsthm,amsmath,amssymb,amscd,mathrsfs}
\usepackage{mathtools,esint}
\allowdisplaybreaks
\usepackage{latexsym}
\usepackage[colorlinks=true,linkcolor=blue,citecolor=red,urlcolor=black]{hyperref}
\usepackage{graphicx}
\usepackage{indentfirst}
\usepackage{cite}
\usepackage{color,xcolor}
\usepackage{lmodern}
\usepackage{booktabs,array,needspace}

\usepackage[
   letterpaper,
   textheight=8.35in,
   headsep=20pt,
   footskip=36pt,
   marginparwidth=0pt,
   marginparsep=0pt,
   left=0.85in,
   right=0.85in
]{geometry}

\hypersetup{
  pdftitle={Regular Solutions of the Stationary Navier--Stokes Equations in Arbitrarily High Dimensions},
  pdfauthor={Wendong Wang and Guoxu Yang},
  pdfsubject={Stationary Navier--Stokes equations in arbitrary fixed high dimension}
}

\newtheorem{theorem}{Theorem}[section]
\newtheorem{lemma}[theorem]{Lemma}
\newtheorem{proposition}[theorem]{Proposition}
\newtheorem{corollary}[theorem]{Corollary}

\theoremstyle{remark}
\newtheorem{remark}[theorem]{Remark}
\newtheorem*{remark*}{Remark}
\numberwithin{equation}{section}

\newif\ifshowchanges
\showchangesfalse

\newcommand{\R}{\mathbb R}
\newcommand{\T}{\mathbb T}
\newcommand{\Z}{\mathbb Z}
\newcommand{\Rn}{\R^n}
\newcommand{\Tn}{\T^n}
\newcommand{\BMO}{\mathrm{BMO}}
\newcommand{\Hdot}{\dot H^1}

\newcommand{\Div}{\operatorname{div}}
\newcommand{\supp}{\operatorname{supp}}
\newcommand{\esssup}{\operatorname*{ess\,sup}}
\newcommand{\loc}{\mathrm{loc}}
\newcommand{\avint}{\fint}
\DeclarePairedDelimiter{\norm}{\lVert}{\rVert}

\newcommand{\proofstep}[2]{\par\medskip\needspace{3\baselineskip}\noindent\textbf{Step #1. #2}\par\smallskip}

\title[Regular solutions of the stationary Navier--Stokes equations]{Regular Solutions of the Stationary Navier--Stokes Equations in Arbitrarily High Dimensions}

\author{Wendong~Wang$^{1}$ \and
        Guoxu~Yang$^{1,*}$}

\thanks{$^{1}$School of Mathematical Sciences, Dalian University of Technology,
Dalian 116024, China.}

\thanks{E-mail address: 
        wendong@dlut.edu.cn (Wendong~Wang),
        guoxu\_dlut@outlook.com (Guoxu~Yang).}

\thanks{$^{*}$Author to whom correspondence should be addressed.}


\begin{document}

\begin{abstract}
Recently, the existence of regular  solutions $(u,p)$ to the stationary incompressible
Navier--Stokes equations on $\R^n$ with bounded compactly supported forces $f$ for
$5\le n\le15$ was established by Li--Yang (Comm. Math. Phys., 2022). However, it is still unknown whether there exists a regular solution for dimensions $n>15$. Here we answer this question and obtain the existence of regular solutions to the stationary Navier--Stokes equations in arbitrarily high dimensions. 
The key innovation is to combine an averaged signed pressure-potential identity with $L^2$-BMO estimates for the Bernoulli source flux, yielding a Bernoulli supremum bound independent of the size of the skew coefficient representing the drift. Together with the energy estimate for the comparison solution, this yields a sublinear bound for the Newtonian potential of the positive Bernoulli function, removing the upper dimension restriction in the earlier existence argument.
We also construct
periodic solutions and whole-space solutions for forces in
$L^\infty\cap L^{2n/(n+2)}$. For compactly supported forces, we record
two far-field profiles, their velocity, gradient and pressure remainders,
and the corresponding cancellation criteria.
\end{abstract}

\maketitle

\tableofcontents

\section{Introduction}\label{sec:results}

Consider the stationary incompressible Navier--Stokes equations as follows:
\begin{equation}\label{eq:NS}
 -\Delta u+(u\cdot\nabla)u+\nabla p=f,\qquad \Div u=0.
\end{equation}

A \emph{regular solution} has $u\in W^{2,q}_{\loc}$ and
$p\in W^{1,q}_{\loc}$ for every $1<q<\infty$, and satisfies
\eqref{eq:NS} in distributions. We write
\begin{equation}\label{eq:Bernoulli-definition}
 \theta=p+\frac{|u|^2}{2},\qquad
 \theta_+=\max\{\theta,0\}.
\end{equation}
In the main existence statements, $n\ge16$ is an arbitrary fixed integer; constants may depend on $n$.

\subsection{Background}
The energy method supplies the natural starting point for stationary
Navier--Stokes theory, but its strength depends on the dimension.
In dimension four the Sobolev estimate $H^1\hookrightarrow L^4$ is
compatible with the quadratic convection term. The regularity results of
Gerhardt~\cite{G1979} and Giaquinta--Modica~\cite{GM1982} belong to this
critical-dimensional theory. For $n\ge5$, the energy embedding gives
only $u\in L^{2n/(n-2)}$, with $2n/(n-2)<n$. Thus it does not by itself
provide a scale-critical $L^n$ bound for the velocity. Constructing one
regular solution for large forces and proving regularity of every energy
weak solution are consequently distinct problems.

Bernoulli estimates have been central to progress on the existence
problem. Frehse and R\r{u}\v{z}i\v{c}ka developed weighted estimates and
interior regularity methods for stationary systems in bounded domains
\cite{FR1994a,FR1994b,FR1997}. Struwe~\cite{Struwe1995} obtained
regular solutions on $\R^5$ and the five-dimensional torus.
The periodic theory of Frehse--R\r{u}\v{z}i\v{c}ka
\cite{FR1995a,FR1995b} reaches dimensions $5\le n\le15$.
These works show why the positive part of the Bernoulli function is
useful: it links scalar elliptic estimates to local velocity control,
even when an energy estimate alone does not close the nonlinear system.
Small-force existence in arbitrary dimensions is a separate line of
work; see, for instance, Farwig--Sohr~\cite{FS2009}.

For the whole-space problem with bounded compactly supported forces,
Li--Yang~\cite[Theorem~1.1]{LY2022} extended Struwe's result to
$5\le n\le15$ and established the Stokes decay rates
$|u(x)|=O(|x|^{2-n})$ and
$|\nabla u(x)|+|p(x)|=O(|x|^{1-n})$.
Their construction combines a weighted Stokes integral equation with
a priori estimates and Leray--Schauder degree. The dimension restriction
comes from the global pressure--velocity estimate in their Lemma~2.11:
the auxiliary exponent must satisfy
\begin{equation}\label{eq:LY-exponent-range}
 \max\{2,n/4\}<q<\min\{4,n/2\}.
\end{equation}
The lower bound $q>n/4$ makes the Bernoulli exponent $nq/(n-2q)$
strictly greater than $n/2$, as required by their local criterion.
The upper bound $q<4$ enters the H\"older estimate for the nonlinear
pressure test. The interval is empty for $n\ge16$; see
\cite[Lemma~2.11 and Remark~2.12]{LY2022}. Motivated by \cite{LY2022}, we consider the case of $n>15$
by replacing its global
integrability argument by a source-flux estimate and scalar comparison.
The local criterion~\cite[Proposition~2.3]{LY2022}, which is stated
without the restriction $n\le15$, remains an explicit input.
The distinction between constructing regular solutions and studying
arbitrary weak solutions also matters in relation to Luo's
high-dimensional stationary nonuniqueness construction~\cite{Luo2019}. 
Substantial progress has been made in the partial regularity theory
for suitable weak solutions of the stationary incompressible
Navier--Stokes equations. In dimension five, Struwe~\cite{Struwe1988}
established partial regularity and, in particular, proved regularity
near $x_0$ under the smallness condition
\begin{equation*}
    \limsup_{r\downarrow0} r^{-1}
    \int_{B_r(x_0)} |\nabla u|^2\,dx < \varepsilon_0.
\end{equation*}
The corresponding boundary theory was subsequently developed by
Kang~\cite{Kang2004}. Dimension six is critical for the standard
energy method: the Sobolev embedding
$H^1(B_1)\hookrightarrow L^3(B_1)$, with $B_1\subset\R^6$,
is continuous but not compact. Dong--Strain~\cite{DongStrain2012}
established interior partial regularity in this critical dimension.
Controlling the pressure remains a central difficulty in establishing
scale-invariant regularity criteria. More recently,
Li--Wang~\cite{LiWang2023} proved interior and boundary one-scale
regularity criteria in dimension six. In particular, their interior
velocity criterion removes the pressure-smallness assumption by
combining a localized pressure decomposition with an inductive
iteration. For further developments, we refer to
\cite{LiuWang2018,C2023,BGLWX2025,CWWYY2026} and the references therein.
Whether suitable weak solutions in dimensions five and six are
necessarily regular remains a major open problem.
The theorems below concern the existence of regular solutions
constructed through a continuation argument.


The behavior at infinity is a further part of the problem.
Jia--\v{S}ver\'ak~\cite{JS2018} studied high-dimensional stationary
asymptotics, while Kang--Miura--Tsai~\cite{KMT2018} described the
Stokes derivative profiles and moment expansions for fast decaying flows.
Once the large-force solutions and their uniform decay have been
constructed, this linear asymptotic machinery becomes available.
The third theorem below makes the resulting coefficients and
cancellations explicit.

\subsection{Main results}
One of the main results is stated as follows.
\begin{theorem}\label{thm:whole}
Let $n\ge16$, let $f\in L^\infty(\Rn;\Rn)$, and suppose
$\supp f\subset B_{R_0}(0)$ with $R_0\ge1$. Equation \eqref{eq:NS} has a regular
solution satisfying
\begin{equation}\label{eq:main-decay}
 |u(x)|\le C_f(1+|x|)^{2-n},\qquad
 |\nabla u(x)|+|p(x)|\le C_f(1+|x|)^{1-n}.
\end{equation}
Here $C_f$ depends only on $n$, $R_0$, and $\|f\|_\infty$. The pressure is the
decaying Stokes pressure specified in Section~\ref{sec:whole-apriori}. Moreover, 
smooth forces give smooth solutions.
\end{theorem}

\begin{remark}
Theorem~\ref{thm:whole}, together with Li--Yang's theorem for
$5\le n\le15$, gives regular-solution existence with Stokes decay
for every fixed $n\ge5$. The obstruction in the earlier proof is
\eqref{eq:LY-exponent-range}, rather than a restriction in the local
Bernoulli criterion. Here the signed identity and the $L^2$--BMO bounds
for the source flux yield the sublinear estimate (see
\eqref{eq:closure-improved}), which supplies the missing a priori bound
in the higher dimensions. The final continuation argument remains the
weighted Leray--Schauder construction (see more details in
Section~\ref{sec:existence} using Theorem~\ref{thm:LS-continuation}).
\end{remark}

The periodic existence results of \eqref{eq:NS} and the solutions with noncompact forces are also established. 
\begin{theorem}\label{thm:extensions}
For every fixed $n\ge16$, the following assertions hold.
\begin{enumerate}
\item If $f\in L^\infty(\Tn;\Rn)$ has mean zero on the unit flat torus,
there is a regular solution of \eqref{eq:NS} with
$\int_{\Tn}u=\int_{\Tn}p=0$. For every $1<q<\infty$,
\begin{equation}\label{eq:torus-main}
 \|u\|_{W^{2,q}(\Tn)}+\|p\|_{W^{1,q}(\Tn)}
 \le C_{n,q}(\|f\|_\infty).
\end{equation}
\item If $f\in L^\infty(\Rn;\Rn)\cap L^{2n/(n+2)}(\Rn;\Rn)$, there is a
regular solution with
\begin{equation}\label{eq:noncompact-main}
\begin{split}
 &\|\nabla u\|_2+\|u\|_{2n/(n-2)}+\|u\|_\infty+\|\nabla u\|_\infty\\
 &\quad+\|\theta_+\|_{2n/(n-2)}+\|\theta_+\|_\infty+\|(-\Delta)^{-1}\theta_+\|_\infty
 \le C_n(N),
\end{split}
\end{equation}
where $N=\|f\|_\infty+\|f\|_{2n/(n+2)}$, and $\theta$ is defined in \eqref{eq:Bernoulli-definition}. The pressure is normalized by
\begin{equation}\label{eq:noncompact-pressure}
 \begin{split}
 p&=p_N+p_f,\\
 p_N&=\partial_i\partial_j(-\Delta)^{-1}(u_i u_j)\in L^{n/(n-2)},\qquad
 p_f=-\partial_j(-\Delta)^{-1}f_j\in L^2.
 \end{split}
\end{equation}
Moreover,
\begin{equation}\label{eq:noncompact-decay-energy}
 \begin{gathered}
 |u(x)|+|\nabla u(x)|+|p(x)|\longrightarrow0\quad (|x|\to\infty),\\
 \|\nabla u\|_2^2=\int_{\Rn}f\cdot u.
 \end{gathered}
\end{equation}
\end{enumerate}
\end{theorem}

\begin{remark}
The periodic existence results of Frehse--R\r{u}\v{z}i\v{c}ka
\cite{FR1995a,FR1995b} cover $5\le n\le15$.
In extending the flux argument to higher dimensions, a positive source
has nonzero mean and cannot be inverted by the mean-zero Laplacian.
Equations~\eqref{eq:torus-zero-mode}--\eqref{eq:zero-mode-absorb}
retain this mean and absorb the resolvent contribution, yielding
\eqref{eq:torus-intermediate}.
For noncompact forces, the constants in a compact-support theorem may
grow with the support radius. Li--Yang~\cite[Remark~1.2]{LY2022}
also discuss sufficiently decaying forces; the assumption in part~(2)
is instead a boundedness and integrability condition with no prescribed
pointwise decay. Proposition~\ref{prop:independent} supplies estimates
independent of the support, and the pressure split in
\eqref{eq:noncompact-pressure} allows the approximation to converge
with its normalization intact. The decay asserted in part~(2) is
qualitative.
\end{remark}

These are existence statements for regular solutions without a smallness
assumption. They do not assert uniqueness, regularity of every energy weak
solution, or a result for domains with boundary. Compact support is used for
the power decay in \eqref{eq:main-decay}; the noncompact-force statement gives
qualitative decay. In particular, no pointwise decay of the noncompact force
is required.

\begin{theorem}\label{thm:far-field}
Under the hypotheses of Theorem~\ref{thm:whole}, let $(u,p)$ be a solution obtained by its weighted construction, with the pressure normalized as in~\eqref{eq:main-decay}. Write $\omega_n=|S^{n-1}|$, $r=|x|$, $e=x/r$, and define
\begin{equation}\label{eq:far-moments}
 b_j=\int_{\Rn}f_j(y)\,dy,\qquad
 A_{kj}=\int_{\Rn}y_kf_j(y)\,dy+\int_{\Rn}u_k(y)u_j(y)\,dy.
\end{equation}
These integrals converge absolutely. For $e\in S^{n-1}$, let $Q_A(e)=n e^{\mathsf T}Ae-\operatorname{tr}A$. Define
\begin{gather}
 \begin{aligned}
 u^{(0)}(x)&=\frac{r^{2-n}}{2(n-2)\omega_n}
 \bigl[b+(n-2)(b\cdot e)e\bigr],\\
 p^{(0)}(x)&=\frac{b\cdot e}{\omega_n r^{n-1}},
 \end{aligned}\label{eq:far-monopole}\\
 \begin{aligned}
 u^{(1)}(x)&=\frac{r^{1-n}}{2\omega_n}
 \bigl[(A^{\mathsf T}-A)e+Q_A(e)e\bigr],\\
 p^{(1)}(x)&=\frac{Q_A(e)}{\omega_n r^n}.
 \end{aligned}\label{eq:far-dipole}
\end{gather}
Then
\begin{equation}\label{eq:far-expansion}
 u=u^{(0)}+u^{(1)}+R_u,\qquad p=p^{(0)}+p^{(1)}+R_p,
\end{equation}
where, for $r\ge2(R_0+1)$,
\begin{equation}\label{eq:far-remainder}
 |R_u(x)|+r|\nabla R_u(x)|+r|R_p(x)|\le C_f r^{-n}.
\end{equation}
In particular,
\begin{equation}\label{eq:far-leading-remainder}
 |u-u^{(0)}|+r|\nabla(u-u^{(0)})|+r|p-p^{(0)}|\le C_f r^{1-n}.
\end{equation}
Here $C_f$ depends only on $n,R_0$, and $\norm{f}_\infty$. The following consequences hold, with all asymptotic statements uniform in direction.
\begin{enumerate}
\item If $b\ne0$, there are positive dimensional constants $c_n,C_n$ and a radius $R_b$ such that
\[
 c_n|b|r^{2-n}\le |u(x)|\le C_n|b|r^{2-n}\qquad(r\ge R_b).
\]
The radius may depend on the data and on $|b|$.
\item One has $b=0$ if and only if $u(x)=o(r^{2-n})$. In this case,
\begin{equation}\label{eq:far-zero-force}
 |u(x)|+r\bigl(|\nabla u(x)|+|p(x)|\bigr)\le C_f r^{1-n}.
\end{equation}
\item If $b=0$, then $u(x)=O(r^{-n})$ if and only if $A=cI$ for some $c\in\R$. Under this condition,
\begin{equation}\label{eq:far-scalar-moment}
 |u(x)|+r\bigl(|\nabla u(x)|+|p(x)|\bigr)\le C_f r^{-n}.
\end{equation}
\end{enumerate}
In particular, a compactly supported force satisfying $\Div f=0$ in distributions has $b=0$ and yields~\eqref{eq:far-zero-force}.
\end{theorem}

\begin{remark}
The Stokes derivative profiles and their scalar-matrix ambiguity are
already present in~\cite[Lemma~4.1, Remark~4.2 and
Proposition~4.5]{KMT2018}. The Navier--Stokes statement in
\cite[Theorem~4.8]{KMT2018} imposes smallness and fast decay.
Theorem~\ref{thm:whole} provides the decay and moment integrability
needed here for arbitrary bounded compactly supported forces, so no
additional smallness is needed in Theorem~\ref{thm:far-field}.
The contribution of this application is the joint velocity, gradient
and pressure remainder estimate and the explicit cancellation tests
for these constructed solutions. When $b=0$ and $A\notin\R I$, the dipole is nonzero as a
profile but may vanish in individual directions.
\end{remark}

\subsection{Notation and the closure estimates}
For $n\ge3$ we use
\[
 \Hdot(\Rn)=\{w\in L^{\frac{2n}{n-2}}(\Rn):\nabla w\in L^2(\Rn)\},
 \qquad \|w\|_{\Hdot}=\|\nabla w\|_2.
\]
This is the homogeneous Sobolev space with its specified $L^{2^*}$
representative, not a space with an arbitrary additive constant. It is the
completion of $C_c^\infty$ for the gradient norm. If $F$ is a locally
integrable scalar or vector field, let
\[
 [F]_{\BMO}:=\sup_B\avint_B|F-F_B|,
 \qquad F_B=\avint_B F.
\]
Componentwise seminorms are equivalent, with dimensional constants. Sobolev
spaces are denoted by $W^{k,q}$; BMO is a space of bounded mean oscillation,
not a space of essentially bounded functions. We use the standard conjugate exponents
\begin{equation}\label{eq:exponents-basic}
 q_*:=\frac{2n}{n+2},\qquad 2^*:=\frac{2n}{n-2},
 \qquad \frac1{q_*}+\frac1{2^*}=1.
\end{equation}
The inverse $(-\Delta)^{-1}$ denotes convolution on $\Rn$ with
$\Gamma(x)=|x|^{2-n}/((n-2)\omega_n)$, where
$\omega_n=|S^{n-1}|$. Periodic inverses
will be defined separately on mean-zero distributions.
For the Bernoulli function in \eqref{eq:Bernoulli-definition}, let
\begin{equation}\label{eq:KMZ}
 K:=\|(-\Delta)^{-1}\theta_+\|_\infty,\qquad
 M:=\sup_{z\in\Rn,\ r>0}r^{2-n}\int_{B_r(z)}|u|^2.
\end{equation}
All unmarked whole-space norms are over $\Rn$.
The notation $C_u$ allows dependence on an individual weighted fixed
point, whereas $C_f$ depends only on the data in the theorem under
consideration. Individual finiteness is established before any uniform
estimate is used.

For the homotopy parameter $0\le\tau\le1$, the source flux is
\begin{equation*}
 F=\nabla(-\Delta)^{-1}(\tau|f||u|)+\tau f.
\end{equation*}
The signed identity gives $M\le C_f(1+K)$, and the flux bounds give
\begin{equation*}\label{eq:intro-H}
 \|\theta_+\|_\infty\le C_f(1+K)^{(n-2)/(2n)},\qquad
 \|\theta_+\|_{2^*}\le C_f.
\end{equation*}
Thus the whole-space potential closes as
\begin{equation*}
 K\le C_f\|\theta_+\|_\infty^{(n-6)/(n-2)}
   \le C_f(1+K)^{(n-6)/(2n)}.
\end{equation*}
The periodic argument retains the constant mode and uses instead
\begin{equation*}
 K\le C_f(1+K)^{(n-4)/(2n)}.
\end{equation*}
Both exponents are less than one for every dimension considered here.
\subsection{Main ideas}
\label{subsec:main-ideas}
The proof of Theorem~\ref{thm:whole} retains the weighted Stokes integral
formulation and Leray--Schauder construction of Li--Yang~\cite{LY2022}. We first establish a priori estimates for the
homotopy fixed points and close the Newtonian potential $K$ of the
positive Bernoulli function. The existing local regularity criteria,
energy tightness, and a finite Stokes bootstrap then yield the uniform
weighted bound required for degree theory. The main new ingredient is
this global a priori closure and the argument concerns the construction of
regular solutions. 
The dimension restriction in the earlier argument arises after the basic
energy estimate. Its Bernoulli estimate~\cite[eq.~(2.11)]{LY2022}
controls $\theta_+$ in $L^{r_q}$ in terms of a local $L^q$ norm of $u$,
where $r_q=nq/(n-2q)$. The local criterion requires $r_q>n/2$, or
$q>n/4$. The accompanying velocity estimate
\cite[Lemma~2.11, eq.~(2.19)]{LY2022}, however, requires the range
in~\eqref{eq:LY-exponent-range}, namely
\[
 \max\{2,n/4\}<q<\min\{4,n/2\}.
\]
More specifically, the H\"older step associated with their eq.~(2.21),
used to estimate
$$(-\Delta)^{-1}(|p|^{(q-2)/2}\operatorname{sgn}(p)\eta)$$
uses the conjugate exponents $2/(q-2)$ and $2/(4-q)$, and hence requires
$2<q<4$. The admissible interval is empty for $n\ge16$; this obstructs
the use of their estimate (2.19), not the Bernoulli identity itself;
see~\cite[Remark~2.12]{LY2022}. Even an extension to $q=4$ would give
only $r_q=n/2$ when $n=16$. Our new observations are as follows.

\paragraph{\textbf{(i) A signed pressure-potential identity and center averaging.}} Different from \cite{LY2022}, we introduce a new quantity, the positive Bernoulli potential
$K=\|(-\Delta)^{-1}\theta_+\|_\infty$, to connect the
signed pressure identity with the source-flux estimates
and close the global a priori bound.
Lemma~\ref{lem:signed} combines the normalized pressure with the kinetic
energy through the exact identity~\eqref{eq:signed}:
\[
 (-\Delta)^{-1}\theta(y)
 =\frac1{2\omega_n}\int_{\Rn}
   \frac{|u(x)\cdot(x-y)|^2}{|x-y|^n}\,dx+\mathcal F_\tau(y),
 \qquad \|\mathcal F_\tau\|_\infty\le C_f.
\]
Since $\theta\le\theta_+$ and the Newtonian kernel is positive, the
directional integral is bounded by $C_f(1+K)$. Averaging over its center
recovers all velocity components and gives
$M\le C_f(1+K)$; see~\eqref{eq:spherical-average}--\eqref{eq:Morrey-bound}.
This reflects the endpoint cancellation in the weighted pressure
identity~\cite[eq.~(2.17)]{LY2022}: the isotropic coefficient
$s(n-4-s)/2$ vanishes at $s=n-4$, while the directional square survives.
Here the endpoint identity is derived directly from the kernels
in~\eqref{eq:GammaPsi}, and center averaging recovers the Morrey control.
Thus no extension of the earlier nonlinear pressure test to $q\ge4$
is needed.

\paragraph{\textbf{(ii) Complementary estimates for the Bernoulli source flux.}}
At the global a priori stage, we estimate the source flux directly. Let
\[
 g=\tau|f||u|,\qquad
 F=\nabla(-\Delta)^{-1}g+\tau f.
\]
The Bernoulli equation takes the form
\[
 L_d\theta=-\Div F-\mu,\qquad 0\le\mu\in L^1,
\]
where $\mu$ is given by~\eqref{eq:comparison-mu} and
$L_d=-\Div((I+d)\nabla)=-\Delta+u\cdot\nabla$.
The skew potential $d$, with $d^{\mathsf T}=-d$, is bounded for each
individual fixed point. Energy gives
\[
 \|g\|_{q_*}\le\|f\|_{n/2}\|u\|_{2^*}\le C_f,
\]
whereas Morrey control gives
$\int_{B_r(z)}g\le C_fM^{1/2}r^{n-1}$.
Lemma~\ref{lem:newton-flux} therefore yields
\[
 \|F\|_2\le C_f,\qquad
 [F]_{\BMO}\le C_f(1+K)^{1/2},
\]
as recorded in~\eqref{eq:F-L2} and~\eqref{eq:F-BMO}.
The first bound contains no power of $K$. The BMO estimate concerns the
\emph{source flux $F$}, rather than the drift potential $d$, and the
derivative of the bounded force remains in divergence form.

\paragraph{\textbf{(iii) Scalar comparison and sublinear potential closure.}}
Let $v\in\Hdot$ solve $L_dv=-\Div F$. The scalar estimate and comparison
principle (Lemmas~\ref{lem:flux-sup} and~\ref{lem:comparison}) give
\[
 \theta\le v,\qquad
 \|\theta_+\|_\infty\le\|v\|_\infty
 \le C_n\|F\|_2^{2/n}[F]_{\BMO}^{1-2/n}
 \le C_f(1+K)^{(n-2)/(2n)}.
\]
The scalar constant is independent of $\|d\|_\infty$, because the skew
part vanishes in every truncation energy test. The energy estimate for
$v$, together with $\theta_+\le v_+$, also gives
$\|\theta_+\|_{2^*}\le C_f$ without any power of $K$;
see~\eqref{eq:comparison-energy}--\eqref{eq:h-2star} in
Proposition~\ref{prop:head}. Since $n\ge16>6$, Lemma~\ref{lem:potential}
applies with exponent $2^*$ and yields
\[
 K\le C_n\|\theta_+\|_{2^*}^{4/(n-2)}
          \|\theta_+\|_\infty^{(n-6)/(n-2)}
   \le C_f(1+K)^{(n-6)/(2n)}.
\]
Because $(n-6)/(2n)<1$ and $K$ is finite for each fixed point by
Lemma~\ref{lem:initial}, this bounds $K$, $M$, and
$\|\theta_+\|_\infty$ uniformly. The resulting whole-space closure
replaces the pressure--velocity integrability loop without requiring
$n/4<q<4$ or a quantitative heat-kernel estimate for the drift operator.

The new contribution is this combination of the signed pressure-potential
identity, source-flux control, and comparison energy to close the global
potential estimate without an upper dimension restriction.
Sections~\ref{sec:scalar}--\ref{sec:existence}
implement the whole-space proof. Section~\ref{sec:torus} extends the
periodic framework of~\cite{FR1995a,FR1995b}, retaining the mean of the
positive source and the constant mode of the resolvent, with closure
exponent $(n-4)/(2n)$. Sections~\ref{sec:noncompact}
and~\ref{sec:far-field} treat support-independent approximation and
far-field expansions.

\section{A scalar estimate independent of the skew coefficient}\label{sec:scalar}

In this section we establish the scalar estimates
underlying the Bernoulli argument. The central result,
Lemma~\ref{lem:flux-sup}, converts $L^2$ and BMO control of a source
flux into an $L^\infty$ bound for the corresponding energy solution,
with a constant independent of the size of the bounded skew coefficient.
This independence is essential: the drift potential associated with
each homotopy fixed point is bounded, but its norm is not controlled
uniformly at this stage. 
\begin{lemma}\label{lem:interpolation}
For $2<q<\infty$ and $F\in L^2(\Rn)\cap\BMO(\Rn)$,
\begin{equation}\label{eq:interpolation}
 \|F\|_q\le C_{n,q}\|F\|_2^{2/q}[F]_{\BMO}^{1-2/q}.
\end{equation}
If $F\in L^2(\Tn)$ and its periodic extension $F_{\mathrm{per}}$ belongs to
$\BMO(\Rn)$, then
\begin{equation}\label{eq:interpolation-torus}
 \|F\|_{L^q(\Tn)}
 \le C_{n,q}\|F\|_{L^2(\Tn)}^{2/q}
 \bigl([F_{\mathrm{per}}]_{\BMO}+\|F\|_{L^2(\Tn)}\bigr)^{1-2/q}.
\end{equation}
\end{lemma}
\begin{proof}

Formula \eqref{eq:interpolation} follows from
\cite[Theorem~9.1]{MRR2013} by $L^2\subset L^{2,\infty}$; the Lorentz-space
extension is discussed in~\cite{DHHN2019}.

For the periodic assertion, write $E=\|F\|_{L^2(\Tn)}$ and
$B=[F_{\mathrm{per}}]_{\BMO}$. Choose a fixed smooth compactly supported
function $\eta$ that equals one on a fundamental cube and is supported in a
fixed bounded union of translated cubes. Clearly
$\|\eta F_{\mathrm{per}}\|_2\le C_nE$. We verify
\begin{equation}\label{eq:cutoff-bmo}
 [\eta F_{\mathrm{per}}]_{\BMO}\le C_n(B+E).
\end{equation}
For a ball $B_r(z)$ with $r<1$, compare the averages of $F_{\mathrm{per}}$ on
successive doubled balls until their radius lies in $[1,2]$. Every difference
of two successive averages is bounded by $C_nB$. The final average is bounded
by $C_nE$, by periodicity and Cauchy--Schwarz on a fixed number of cells. Hence
\[
 |(F_{\mathrm{per}})_{B_r(z)}|
 \le C_nE+C_nB\log(2/r).
\]
Use the comparison constant
$c=\eta_{B_r(z)}(F_{\mathrm{per}})_{B_r(z)}$ in the mean oscillation of the
product. Its average deviation from $c$ is at most
\[
 C_nB+C_nr\bigl(E+B\log(2/r)\bigr)\le C_n(B+E).
\]
Replacing $c$ by the actual product average costs at most a factor two. On
balls of radius at least one, compact support of $\eta$ and the bound
$\|\eta F_{\mathrm{per}}\|_1\le C_nE$ prove the same estimate. This proves
\eqref{eq:cutoff-bmo}. Apply the whole-space assertion to
$\eta F_{\mathrm{per}}$ and restrict to the fundamental cube.
\end{proof}

\begin{lemma}\label{lem:flux-sup}
Let $n\ge3$, and let $b\in L^\infty(\Rn;\R^{n\times n})$ be skew-symmetric.
If $F\in L^2(\Rn;\Rn)\cap\BMO(\Rn;\Rn)$, the unique energy solution
$w\in\Hdot(\Rn)$ of
\begin{equation}\label{eq:scalar}
 L_bw:=-\Div((I+b)\nabla w)=-\Div F
\end{equation}
satisfies
\begin{equation}\label{eq:scalar-sup}
 \|w\|_\infty\le C_n\|F\|_2^{2/n}[F]_{\BMO}^{1-2/n}.
\end{equation}
The constant is independent of $\|b\|_\infty$. On $\Tn$, if $b$ is a bounded
periodic skew matrix and $F\in L^2(\Tn)$ has BMO periodic extension, the unique
weak solution $w\in W^{1,2}(\Tn)$ of
\begin{equation}\label{eq:scalar-torus}
 (1+L_b)w=-\Div F
\end{equation}
satisfies
\begin{equation}\label{eq:scalar-sup-torus}
 \|w\|_\infty\le C_n\|F\|_2^{2/n}
 \bigl([F_{\mathrm{per}}]_{\BMO}+\|F\|_2\bigr)^{1-2/n},
\end{equation}
again independently of the size of the bounded skew coefficient.
\end{lemma}
\begin{proof}
 \proofstep{I}{Energy solution and truncation}
The bilinear form
\[
 a(w,\phi)=\int_{\Rn}(I+b)\nabla w\cdot\nabla\phi
\]
is bounded on $\Hdot\times\Hdot$ for the fixed coefficient, and
$a(w,w)=\|\nabla w\|_2^2$. The functional
$\phi\mapsto\int F\cdot\nabla\phi$ is bounded there. Lax--Milgram gives
existence and uniqueness. Writing $E=\|F\|_2$ and $Q=\|F\|_{2n}$, testing by
$w$ and applying Sobolev give
\begin{equation}\label{eq:scalar-energy}
 \|\nabla w\|_2\le E,\qquad \|w\|_{2^*}\le C_nE.
\end{equation}
The number $Q$ is finite by Lemma~\ref{lem:interpolation}.

For $k>0$, $w_k=(w-k)_+$ belongs to $\Hdot$: it is bounded pointwise by
$|w|$, and the Sobolev chain rule gives
$\nabla w_k=\boldsymbol1_{\{w>k\}}\nabla w$. The weak formulation extends to
all of $\Hdot$ by density, so $w_k$ is an admissible test. The identity
\[
 b\nabla w\cdot\nabla w_k
 =b\nabla w_k\cdot\nabla w_k=0
 \quad\text{almost everywhere}
\]
gives
\begin{equation}\label{eq:trunc-energy}
 \int|\nabla w_k|^2
 \le\left(\int_{\{w>k\}}|F|^2\right)^{1/2}\|\nabla w_k\|_2,
 \qquad
 \|\nabla w_k\|_2^2\le\int_{\{w>k\}}|F|^2.
\end{equation}

 \proofstep{II}{The level recurrence}
Let $m(k)=|\{w>k\}|$, which is finite for $k>0$ by
\eqref{eq:scalar-energy}. For $\ell>k$, Sobolev and H\"older imply
\begin{equation}\label{eq:level-rec}
 (\ell-k)^2m(\ell)^{(n-2)/n}
 \le C_nQ^2m(k)^{(n-1)/n}.
\end{equation}
Indeed, $w_k\ge\ell-k$ on $\{w>\ell\}$, while
$\int_{\{w>k\}}|F|^2\le Q^2m(k)^{1-1/n}$.

Fix $k_0>0$. If $m(k_0)=0$, there is nothing to iterate. Otherwise write
$m_0=m(k_0)>0$, choose a constant $A_n$ below, and set
\[
 d=A_nQm_0^{1/(2n)},\qquad
 k_j=k_0+d(1-2^{-j}),\qquad Y_j=\frac{m(k_j)}{m_0}.
\]
Put $\gamma=n/(n-2)$ and $\delta=1/(n-2)$. Raising
\eqref{eq:level-rec} to the power $\gamma$ yields
\begin{equation}\label{eq:normalized-rec}
 Y_{j+1}\le C_nA_n^{-2\gamma}
 2^{2\gamma(j+1)}Y_j^{1+\delta}.
\end{equation}
The power of $m_0$ cancels, because $\delta-\gamma/n=0$.
Choose $A_n$ so large that
$C_nA_n^{-2\gamma}2^{2\gamma}\le2^{-2n}$. If
$Y_j\le2^{-2nj}$, then $2\gamma-2n\delta=0$ shows that the right side of
\eqref{eq:normalized-rec} is at most $2^{-2n(j+1)}$. Since $Y_0=1$,
induction proves $Y_j\le2^{-2nj}$. Consequently
$m(k_0+d)=0$ and
\begin{equation}\label{eq:level-end}
 \esssup w\le k_0+C_nQm(k_0)^{1/(2n)}.
\end{equation}

 \proofstep{III}{Optimization and BMO interpolation}
Chebyshev applied to \eqref{eq:scalar-energy}, followed by
\eqref{eq:level-end}, gives
\begin{equation}\label{eq:level-optimization}
 m(k_0)\le(C_nE/k_0)^{2^*},\qquad
 \esssup w\le k_0+C_nQ(E/k_0)^{1/(n-2)}.
\end{equation}
If $EQ>0$, choose
$k_0=Q^{(n-2)/(n-1)}E^{1/(n-1)}$. If $E=0$ or $Q=0$, the energy equation
gives $w=0$. Apply the same argument to $-w$, whose forcing flux is $-F$.
We have proved
\begin{equation}\label{eq:finite-q-sup}
 \|w\|_\infty\le C_nQ^{(n-2)/(n-1)}E^{1/(n-1)}.
\end{equation}
Lemma~\ref{lem:interpolation} with the finite exponent $q=2n$ yields
$Q\le C_nE^{1/n}[F]_{\BMO}^{1-1/n}$. The two resulting powers are
\[
 \frac{n-2}{n(n-1)}+\frac1{n-1}=\frac2n,
 \qquad
 \frac{n-2}{n-1}\left(1-\frac1n\right)=1-\frac2n.
\]
This proves \eqref{eq:scalar-sup}.
In particular, no endpoint embedding of BMO into $L^\infty$ is used:
the only interpolation exponent is $2n<\infty$, and the powers in
\eqref{eq:finite-q-sup} are explicitly balanced in
\eqref{eq:level-optimization}.

 \proofstep{IV}{Periodic version}
Use the coercive form
$\int w\phi+\int(I+b)\nabla w\cdot\nabla\phi$ on $W^{1,2}(\Tn)$.
Testing by $w$ gives $\|w\|_{W^{1,2}}\le E$, hence
$\|w\|_{2^*}\le C_nE$. For $k\ge0$,
$w(w-k)_+\ge(w-k)_+^2$. The truncation test therefore bounds the full
$W^{1,2}$ norm of $(w-k)_+$ by the right side of
\eqref{eq:trunc-energy}. Periodic inhomogeneous Sobolev gives exactly
\eqref{eq:level-rec}. The preceding iteration applies without a change,
and \eqref{eq:interpolation-torus} replaces \eqref{eq:interpolation} in its
last step. This proves \eqref{eq:scalar-sup-torus}.
\end{proof}




\begin{lemma}\label{lem:comparison}
Let $b$ be a bounded skew matrix on $\Rn$, $n\ge3$. If
$r\in\Hdot(\Rn)\cap L^\infty(\Rn)$ and
\[
 L_br=-\mu\quad\text{in distributions},\qquad
 \mu\in L^1(\Rn),\quad\mu\ge0,
\]
then $r\le0$ almost everywhere. On $\Tn$, if
$r\in W^{1,2}(\Tn)\cap L^\infty(\Tn)$ and
$(1+L_b)r=-\mu$ with $0\le\mu\in L^1(\Tn)$, then again $r\le0$.
\end{lemma}
\begin{proof}
Fix $b,r,\mu$ throughout the limiting argument. Choose
$\chi\in C_c^\infty(B_2)$ with $0\le\chi\le1$, $\chi=1$ on $B_1$, and set
$\chi_R(x)=\chi(x/R)$, $A_R=B_{2R}\setminus B_R$.
The function $\chi_R^2r_+$ is nonnegative, bounded, compactly supported, and
belongs to $W^{1,2}(\Rn)$. Convolving it with nonnegative smooth mollifiers
produces smooth nonnegative compactly supported functions, uniformly bounded
by $\|r\|_\infty$, convergent in $W^{1,2}$ and, along a subsequence, almost
everywhere. Boundedness of $I+b$ passes the bilinear form to the limit.
The source term passes by dominated convergence against $\mu\in L^1$.
This justifies the test $\chi_R^2r_+$ without assuming that every $L^1$
function defines a continuous functional on $\Hdot$.

Skew-symmetry and the chain rule give
\begin{equation}\label{eq:comparison-cutoff}
 \int\chi_R^2|\nabla r_+|^2
 =-\int\mu\chi_R^2r_+
 -2\int\chi_Rr_+(I+b)\nabla r\cdot\nabla\chi_R.
\end{equation}
The first term on the right is nonpositive. H\"older on $A_R$, using
$1/2-1/2^*=1/n$, gives
\begin{equation}\label{eq:comparison-error}
\begin{split}
 |\mathrm{Err}_R|
 &\le C(1+\|b\|_\infty)R^{-1}
 \|\nabla r\|_{L^2(A_R)}\|r_+\|_{L^2(A_R)}\\
 &\le C_n(1+\|b\|_\infty)
 \|\nabla r\|_{L^2(A_R)}\|r_+\|_{L^{2^*}(A_R)}\longrightarrow0.
\end{split}
\end{equation}
Both tail norms tend to zero and the coefficient norm is finite for the
fixed solution. Passing to the limit in \eqref{eq:comparison-cutoff} gives
$\|\nabla r_+\|_2=0$. A zero-gradient function is constant; its
$L^{2^*}(\Rn)$ membership forces that constant to be zero.

On the torus, mollify $r_+$ periodically. The same bounded approximation
justifies the test directly and gives
\[
 \int_{\Tn}(r_+^2+|\nabla r_+|^2)
 =-\int_{\Tn}\mu r_+\le0.
\]
The conclusion follows. No limit involving a family of coefficients is used
in either argument.
\end{proof}

\begin{lemma}\label{lem:newton-flux}
Let $n\ge3$, $g\in L^{q_*}(\Rn)$, and
\[
 m_1(g):=\sup_{z\in\Rn,\ r>0}r^{1-n}\int_{B_r(z)}|g|<\infty.
\]
The flux $F_g=\nabla (-\Delta)^{-1}g$ has its $L^2$ representative and satisfies
\begin{equation}\label{eq:newton-flux}
 -\Div F_g=g,\qquad
 \|F_g\|_2\le C_n\|g\|_{q_*},\qquad
 [F_g]_{\BMO}\le C_nm_1(g).
\end{equation}
\end{lemma}
\begin{proof}
The kernel $J=\nabla\Gamma$ obeys
$|J(x)|\le C_n|x|^{1-n}$ and $|\nabla J(x)|\le C_n|x|^{-n}$.
The order-one Hardy--Littlewood--Sobolev inequality
\cite{S1993} gives the $L^2$ estimate, since
$1/q_*-1/n=1/2$. It also ensures local absolute convergence almost everywhere
by applying the inequality to $|g|$. The far convolution is absolutely
convergent by H\"older, because
$|x|^{1-n}\boldsymbol1_{\{|x|>1\}}\in L^{2^*}$.
The identity $-\Div F_g=g$ follows in distributions from
$-\Delta (-\Delta)^{-1}g=g$, first for smooth compact sources and then by the potential
bound and approximation in $L^{q_*}$.

Fix $B=B_r(z)$, and split the convolution at $B_{2r}(z)$. For the near part,
Tonelli and local integrability of the kernel give
\begin{align*}
 \avint_B|F_{g,\mathrm{near}}(x)|\,dx
 &\le C_nr^{-n}\int_{B_{2r}(z)}|g(y)|
       \int_B|x-y|^{1-n}\,dx\,dy\\
 &\le C_nr^{1-n}\int_{B_{2r}(z)}|g|
 \le C_nm_1(g).
\end{align*}

For $x\in B$ and $y\notin B_{2r}(z)$, the mean value theorem gives
$|J(x-y)-J(z-y)|\le C_nr|y-z|^{-n}$. Thus
\begin{align*}
 \sup_{x\in B}|F_{g,\mathrm{far}}(x)-F_{g,\mathrm{far}}(z)|
 &\le C_nr\sum_{j\ge1}(2^jr)^{-n}
          \int_{B_{2^{j+1}r}(z)}|g|\\
 &\le C_nm_1(g)\sum_{j\ge1}2^{-j}.
\end{align*}
Use the constant $F_{g,\mathrm{far}}(z)$ when estimating the mean oscillation
on $B$. Replacing that constant by the average of $F_g$ costs a factor at
most two. Taking the supremum over $B$ proves the BMO assertion.
\end{proof}

\begin{lemma}\label{lem:potential}
Let $n>2$, $1<s<n/2$, and $0\le a\in L^s(\Rn)\cap L^\infty(\Rn)$. Then
\begin{equation}\label{eq:potential-general}
 \|(-\Delta)^{-1}a\|_\infty\le
 C_{n,s}\|a\|_s^{2s/n}\|a\|_\infty^{1-2s/n}.
\end{equation}
In particular, if $n>6$, $s=2^*$ is allowed and
\begin{equation}\label{eq:potential-2star}
 \|(-\Delta)^{-1}a\|_\infty\le
 C_n\|a\|_{2^*}^{4/(n-2)}\|a\|_\infty^{(n-6)/(n-2)}.
\end{equation}
\end{lemma}

\begin{proof}
For $\rho>0$, direct integration on $B_\rho(y)$ and H\"older's
inequality on its complement give
\begin{equation}\label{eq:potential-split-general}
 (-\Delta)^{-1}a(y)
 \le C_n\|a\|_\infty\rho^2
   +C_{n,s}\|a\|_s\rho^{2-n/s}.
\end{equation}
Indeed, if $s'=s/(s-1)$, then
\[
 \int_{|z|>\rho}|z|^{(2-n)s'}\,dz
 =\frac{\omega_n}{(n-2)s'-n}\,
   \rho^{n-(n-2)s'},
\]
and the denominator is positive exactly when $s<n/2$.

If both norms are nonzero, take
$\rho=(\|a\|_s/\|a\|_\infty)^{s/n}$ in
\eqref{eq:potential-split-general}. Both terms have size
$\|a\|_s^{2s/n}\|a\|_\infty^{1-2s/n}$, proving
\eqref{eq:potential-general}. If either norm vanishes, $a=0$ almost
everywhere and the conclusion is immediate.
Finally, $2^*<n/2$ is equivalent to $n>6$, and substitution of
$s=2^*=2n/(n-2)$ gives \eqref{eq:potential-2star}.

\end{proof}

\section{Whole-space a priori bounds}\label{sec:whole-apriori}
Assume that $f$ is bounded and supported in $B_{R_0}(0)$, and we estimate
all weighted fixed points for the homotopy equation
\begin{equation}\label{eq:homotopy}
 -\Delta u+(u\cdot\nabla)u+\nabla p=\tau f,
 \qquad \Div u=0,\qquad 0\le\tau\le1.
\end{equation}
The weighted Banach space is
\begin{equation}\label{eq:X}
 X=\{v\in C^1(\Rn;\Rn):\Div v=0,\ \|v\|_X<\infty\},
\end{equation}
where
\[
 \|v\|_X=\sup_x(1+|x|)^{n-3}|v(x)|
          +\sup_x(1+|x|)^{n-2}|\nabla v(x)|.
\]
Write $\omega_n=|S^{n-1}|$. The Stokes kernels, with this surface-area
normalization, are
\begin{equation}\label{eq:Stokes-kernels}
 U_{ij}(x)=\frac1{2(n-2)\omega_n}
 \left(\delta_{ij}|x|^{2-n}+(n-2)x_ix_j|x|^{-n}\right),
 \qquad P_j(x)=\frac{x_j}{\omega_n|x|^n}.
\end{equation}
They satisfy $-\Delta U_{ij}+\partial_iP_j=\delta_{ij}\delta_0$ and
$\partial_iU_{ij}=0$. For $n>4$, set
\begin{equation}\label{eq:GammaPsi}
 \Gamma(x)=\frac{|x|^{2-n}}{(n-2)\omega_n},\qquad
 \Psi(x)=\frac{|x|^{4-n}}{2(n-4)(n-2)\omega_n}.
\end{equation}
Thus $-\Delta\Gamma=\delta_0$, $-\Delta\Psi=\Gamma$, and $(-\Delta)^{-1}a=\Gamma*a$ whenever this potential is defined.

\begin{lemma}\label{lem:initial}
Let $u\in X$ satisfy
\begin{equation}\label{eq:fixed}
 u_i=U_{ij}*G_j,\qquad G=\tau f-(u\cdot\nabla)u,
 \qquad 0\le\tau\le1,
\end{equation}
and set $p=P_j*G_j$. Then $(u,p)$ is a regular solution of
\eqref{eq:homotopy}, and
\begin{equation}\label{eq:individual-decay}
 |u(x)|\le C_u(1+|x|)^{2-n},\qquad
 |\nabla u(x)|+|p(x)|\le C_u(1+|x|)^{1-n}.
\end{equation}
The pressure has the normalization
\begin{equation}\label{eq:pressure-normalization}
 p=\partial_i\partial_j(-\Delta)^{-1}(u_i u_j)-\tau\partial_i(-\Delta)^{-1}f_i.
\end{equation}
Moreover, $u\in L^2$, $\theta\in\Hdot\cap L^\infty$, $K+M<\infty$, and
\begin{equation}\label{eq:skew-d}
 d_{ij}=\partial_i(-\Delta)^{-1}u_j-\partial_j(-\Delta)^{-1}u_i
\end{equation}
belongs to $W^{1,\infty}(\Rn)$ and satisfies
\begin{equation}\label{eq:drift-representation}
 d_{ij}=-d_{ji},\qquad \partial_jd_{ij}=u_i,
 \qquad -\Div((I+d)\nabla)=-\Delta+u\cdot\nabla.
\end{equation}
All constants in this lemma may depend on the individual fixed point.
\end{lemma}
\begin{proof}
The definition of $X$ gives
$|G(x)|\le C_u(1+|x|)^{5-2n}$, after increasing $C_u$ on the force support.
We shall repeatedly use the elementary convolution bound
\begin{equation}\label{eq:convolution}
 \int_{\Rn}|x-y|^{a-n}(1+|y|)^{-\beta}\,dy
 \le
 \begin{cases}
 C(1+|x|)^{a-\beta},&a<\beta<n,\\
 C(1+|x|)^{a-n},&\beta>n,
 \end{cases}
 \qquad 0<a<n.
\end{equation}
The case $\beta=n$ is not used. To verify the formula, assume $r=|x|\ge2$.
On $|y|<r/2$ the integral is bounded by
$Cr^{a-n}\int_0^{r/2}s^{n-1}(1+s)^{-\beta}\,ds$; this has the two stated
orders. On $|x-y|<r/2$ it is at most $Cr^{-\beta}r^a$. On the remaining
part with $|y|\le2r$, both distances are bounded below by $r/2$, so the
bound is again $Cr^{a-\beta}$. Finally, on $|y|>2r$, integrate
$Cs^{a-\beta-1}$ from $2r$ to infinity. Bounded $x$ follows directly from
local integrability and $\beta>a$.

Since $2n-5>n$, apply \eqref{eq:convolution} with $\beta=2n-5$ and
$a=2,1,1$ to $U$, $\nabla U$, and $P$, respectively. These kernels are
locally integrable, and this proves \eqref{eq:individual-decay}. The
fundamental solution identities imply \eqref{eq:homotopy} in distributions.
The source $G$ is locally bounded. Local Stokes estimates, or the
pressure-free estimate proved in Lemma~\ref{lem:interior-Stokes}, give
$u\in W^{2,q}_{\loc}$ for every finite $q>1$; the equation then gives
$p\in W^{1,q}_{\loc}$.

Individual decay implies $u\in L^2\cap L^{2^*}$, $\nabla u\in L^2$, and
$p\in L^{\frac n{n-2}}\cap L^\infty$. The right side $p_*$ of
\eqref{eq:pressure-normalization} also belongs to $L^{\frac n{n-2}}$. Indeed,
$u_i u_j\in L^{\frac n{n-2}}$ and $\partial_i\partial_j(-\Delta)^{-1}$ is an order-zero
Calder\'on--Zygmund operator. The force term belongs to $L^{\frac n{n-2}}$ by the
order-one potential estimate applied to the compact force in $L^{\frac n{n-1}}$.
Taking divergence of the equation gives
\[
 -\Delta p=\partial_i\partial_j(u_i u_j)-\tau\partial_i f_i=-\Delta p_*.
\]
Thus $p-p_*$ is a harmonic distribution, hence a smooth harmonic function,
in $L^{\frac n{n-2}}(\Rn)$. Its mean-value estimate on arbitrarily large balls
forces it to vanish. This proves \eqref{eq:pressure-normalization} without
leaving an additive constant.

The individual source $G$ is in $L^2$. Singular-integral estimates on the
Stokes pressure give $\nabla p\in L^2$. Since $u$ is individually bounded,
$\nabla\theta=\nabla p+u_j\nabla u_j\in L^2$. Also
$\theta\in L^{\frac n{n-2}}\cap L^\infty\subset L^{2^*}$. Therefore
$\theta\in\Hdot\cap L^\infty$. Equation \eqref{eq:individual-decay} gives
$|\theta(x)|+|p(x)|\le C_u(1+|x|)^{1-n}$. Applying
\eqref{eq:convolution} with $a=2$, $\beta=n-1$ yields
\begin{equation}\label{eq:individual-potentials}
 (-\Delta)^{-1}|\theta|(y)+(-\Delta)^{-1}|p|(y)\le C_u(1+|y|)^{3-n}.
\end{equation}
Hence $K<\infty$ before any uniform estimate is attempted. For $M$, use
$\|u\|_\infty$ on balls of radius at most one and $\|u\|_2$ on larger
balls. This proves its initial finiteness.

The kernel $\nabla\Gamma$ is locally integrable and belongs to
$L^2(\{|x|>1\})$. Consequently
\[
 \|d\|_\infty\le C_n(\|u\|_\infty+\|u\|_2).
\]
Each first derivative of $d$ is an order-zero singular integral of $u$.
In the principal-value integral over $|z|<1$, subtract $u(x)$ and use
$|u(x-z)-u(x)|\le\|\nabla u\|_\infty|z|$. The remaining local integrand
is bounded by $C_n\|\nabla u\|_\infty|z|^{1-n}$, which is integrable.
The distributional local multiple of $u(x)$ is bounded by $C_n\|u\|_\infty$,
and the far integral is bounded by $C_n\|u\|_2$. Thus
\[
 \|\nabla d\|_\infty
 \le C_n(\|\nabla u\|_\infty+\|u\|_\infty+\|u\|_2)<\infty.
\]
Finally, commutation in distributions gives
$\partial_jd_{ij}=\partial_i(-\Delta)^{-1}(\Div u)-\Delta (-\Delta)^{-1}u_i=u_i$.
For a smooth scalar $\phi$, skew-symmetry cancels
$d_{ij}\partial_{ij}\phi$, and $\partial_i d_{ij}=-u_j$ gives
\eqref{eq:drift-representation}. 
\end{proof}

\subsection{Energy, the signed identity, and Morrey control}
The energy identity and the pressure representation give
\begin{equation}\label{eq:basic-energy}
 \|\nabla u\|_2+\|u\|_{2^*}+\|p\|_{\frac n{n-2}}
 +\|\nabla p\|_{\frac n{n-1}}+\|\theta_+\|_{\frac n{n-2}}\le C_f.
\end{equation}
We spell out the cutoff justification. Testing \eqref{eq:homotopy} by
$u\zeta_R$, where $\zeta_R=1$ on $B_R$, vanishes outside $B_{2R}$, and
$|\nabla\zeta_R|\le C/R$, leaves errors bounded by
\[
 \frac C R\int_{B_{2R}\setminus B_R}
  (|u||\nabla u|+|u|^3+|p||u|).
\]
By \eqref{eq:individual-decay}, they are bounded, respectively, by
$C_uR^{2-n}$, $C_uR^{5-2n}$, and $C_uR^{2-n}$, and tend to zero.
Consequently
\begin{equation}\label{eq:energy-equality}
 \|\nabla u\|_2^2=\tau\int f\cdot u
 \le\|f\|_{q_*}\|u\|_{2^*}
 \le C_n\|f\|_{q_*}\|\nabla u\|_2.
\end{equation}
After the limiting identity is justified, its bound uses only the data.
For the pressure, H\"older gives
\[
 \|(u\cdot\nabla)u\|_{\frac n{n-1}}
 \le\|u\|_{2^*}\|\nabla u\|_2,
 \qquad \frac{n-1}{n}=\frac1{2^*}+\frac12.
\]
Since $p=P*G$, the order-one potential bound and the order-zero bound for
$\nabla p$ give the pressure terms in \eqref{eq:basic-energy}. Every required
norm of the compact force is controlled by $n,R_0,\|f\|_\infty$.
Finally $\theta_+\le|p|+|u|^2/2$ gives its $L^{\frac n{n-2}}$ bound. In particular, no norm
of $\nabla f$ has been used.

\begin{lemma}\label{lem:signed}
For the normalized pressure of Lemma~\ref{lem:initial},
\begin{equation}\label{eq:signed}
 (-\Delta)^{-1}\theta(y)=\frac1{2\omega_n}
 \int_{\Rn}\frac{|u(x)\cdot(x-y)|^2}{|x-y|^n}\,dx+\mathcal F_\tau(y),
 \qquad \mathcal F_\tau=-\tau\partial_i\Psi*f_i.
\end{equation}
All displayed integrals converge absolutely. Moreover,
\begin{equation}\label{eq:force-remainder}
 \|\mathcal F_\tau\|_\infty\le C_nR_0^3\|f\|_\infty\le C_f,
 \qquad M\le C_f(1+K).
\end{equation}
\end{lemma}
\begin{proof}
For $x\ne0$, differentiate \eqref{eq:GammaPsi} to obtain
\begin{equation}\label{eq:Hessian-Psi}
 \partial_i\partial_j\Psi(x)
 =-\frac{\delta_{ij}|x|^{2-n}}{2(n-2)\omega_n}
   +\frac{x_ix_j}{2\omega_n|x|^n}.
\end{equation}
These derivatives are locally integrable. There is no extra distribution
supported at zero: the boundary terms in two integrations by parts on
$\{|x|>\varepsilon\}$ are bounded by constants times $\varepsilon^2$ or
$\varepsilon^3$ against a fixed smooth test, and vanish as
$\varepsilon\downarrow0$.

Define
\[
 Q=\partial_i\partial_j\Psi*(u_i u_j)-\tau\partial_i\Psi*f_i.
\]
The velocity product is bounded, integrable, and decays as
$O((1+|x|)^{4-2n})$. The two kernels have orders two and three. Absolute
convergence follows from these facts and compact support of $f$; the same
near/far estimates give $|Q(y)|\le C_u(1+|y|)^{3-n}$. Distributional
differentiation and \eqref{eq:pressure-normalization} imply $-\Delta Q=p$.
By \eqref{eq:individual-potentials}, $(-\Delta)^{-1}p$ is also absolutely defined,
tends to zero at infinity, and satisfies $-\Delta (-\Delta)^{-1}p=p$. Thus
$Q-(-\Delta)^{-1}p$ is harmonic and tends to zero at infinity. The maximum principle
on expanding balls, or the mean-value property, makes it zero. This
identifies the twice-integrated pressure without an exchange of singular
convolutions or an unidentified harmonic remainder.

Adding $(-\Delta)^{-1}(|u|^2/2)$ to $Q$, the first term in
\eqref{eq:Hessian-Psi} cancels exactly, proving \eqref{eq:signed}. Its
remaining velocity integral is nonnegative and absolutely convergent, with
local kernel bounded by $|x-y|^{2-n}|u(x)|^2$. For the force term,
$|\nabla\Psi(z)|\le C_n|z|^{3-n}$, and
\[
 \sup_y\int_{B_{R_0}}|x-y|^{3-n}\,dx\le C_nR_0^3.
\]
If $|y|\le2R_0$, enclose the integration set in $B_{3R_0}(y)$ and integrate
radially; otherwise use $|x-y|\ge|y|/2$. This proves its data bound.

Let
$\mathcal R(y)=\int|u(x)\cdot(x-y)|^2|x-y|^{-n}\,dx$.
Positivity of $\Gamma$, the inequality $\theta\le \theta_+$, and
\eqref{eq:signed} imply $0\le\mathcal R(y)\le C_f(1+K)$. Fix a center $z$
and radius $r>0$. For $x\in B_r(z)$, $B_r(x)\subset B_{2r}(z)$, so Tonelli
and spherical integration give
\begin{align}
 C_f(1+K)r^n
 &\ge\int_{B_{2r}(z)}\mathcal R(y)\,dy\notag\\
 &\ge\int_{B_r(z)}\int_{B_r(x)}
      \frac{|u(x)\cdot(x-y)|^2}{|x-y|^n}\,dy\,dx
 =\frac{\omega_nr^2}{2n}\int_{B_r(z)}|u(x)|^2\,dx.
 \label{eq:spherical-average}
\end{align}
The last equality uses
$\int_{S^{n-1}}|a\cdot\omega|^2\,d\omega=\omega_n|a|^2/n$ and
$\int_0^r s\,ds=r^2/2$. Dividing by $r^n$ and taking the supremum gives
\begin{equation}\label{eq:Morrey-bound}
 M\le C_f(1+K).
\end{equation}
\end{proof}

\subsection{Estimating the Bernoulli source as a flux}
Use the individually bounded skew potential $d$ from Lemma~\ref{lem:initial}
and write $L_d=-\Delta+u\cdot\nabla$. The Bernoulli identity is
\begin{equation}\label{eq:Bernoulli}
 L_d\theta=-D(u)+\tau f\cdot u-\tau\Div f,
 \qquad D(u)=\frac12\sum_{i,j}(\partial_i u_j-\partial_j u_i)^2\ge0.
\end{equation}
For completeness, the velocity equation dotted with $u$ gives
\[
 -\Delta\frac{|u|^2}{2}+u\cdot\nabla\theta
 =\tau f\cdot u-|\nabla u|^2.
\]
Taking its divergence instead gives
$-\Delta p=\partial_i u_j\partial_j u_i-\tau\Div f$.
Add these two identities and use
$D(u)=|\nabla u|^2-\partial_i u_j\partial_j u_i$.
All products are legitimate locally for the regular solution, and the
result holds in distributions for bounded $f$. The derivative of $f$
remains in divergence form throughout.

\begin{proposition}\label{prop:head}
Every weighted fixed point satisfies
\begin{equation}\label{eq:head-bound}
 \|\theta_+\|_\infty\le C_f(1+K)^{(n-2)/(2n)},\qquad
 \|\theta_+\|_{2^*}\le C_f.
\end{equation}
More precisely, with
\begin{equation}\label{eq:source-flux}
 g=\tau|f||u|,\qquad F_g=\nabla (-\Delta)^{-1}g,\qquad F=F_g+\tau f,
\end{equation}
there is an energy function $v$ satisfying $L_dv=-\Div F$ such that
$\theta\le v$, $\|\nabla v\|_2+\|v\|_{2^*}\le C_f$, and
$\|v\|_\infty\le C_f(1+K)^{(n-2)/(2n)}$.
\end{proposition}
\begin{proof}
There are two different estimates for $F$. First,
\begin{equation}\label{eq:g-qstar}
 \|g\|_{q_*}\le\|f\|_{n/2}\|u\|_{2^*}\le C_f,
 \qquad \frac1{q_*}=\frac2n+\frac1{2^*}.
\end{equation}
The energy bound, not an unknown supremum of $u$, is used here.
The exponent identity is
\[
 \frac{n+2}{2n}=\frac2n+\frac{n-2}{2n};
\]
hence the first inequality of \eqref{eq:g-qstar} is exactly H\"older
with $f\in L^{n/2}$ and $u\in L^{2^*}$. Lemma~\ref{lem:newton-flux} and $f\in L^2$ give
\begin{equation}\label{eq:F-L2}
 \|F\|_2\le C_n\|g\|_{q_*}+\|f\|_2\le C_f.
\end{equation}

Second, Cauchy--Schwarz on each ball yields
\begin{equation*}
 \int_{B_r(z)}g
 \le\|f\|_\infty|B_r|^{1/2}
       \left(\int_{B_r(z)}|u|^2\right)^{1/2}
 \le C_n\|f\|_\infty M^{1/2}r^{n-1}.
\end{equation*}
Consequently $m_1(g)\le C_fM^{1/2}$. Since a bounded field has BMO
seminorm at most twice its supremum norm,
\begin{equation}\label{eq:F-BMO}
 [F]_{\BMO}\le C_nm_1(g)+2\|f\|_\infty
 \le C_f(1+M^{1/2})\le C_f(1+K)^{1/2}.
\end{equation}
The flux $F$, not the skew matrix $d$, is the BMO quantity in this estimate.

Let $v\in\Hdot$ be the unique solution of $L_dv=-\Div F$. The coefficient
$d$ is bounded for this fixed point, so Lemma~\ref{lem:flux-sup} applies.
The flux bounds \eqref{eq:F-L2} and \eqref{eq:F-BMO}, inserted
into the scalar estimate \eqref{eq:scalar-sup}, give
\[
 \|v\|_\infty
 \le C_n\|F\|_2^{2/n}[F]_{\BMO}^{1-2/n}
 \le C_f\bigl((1+K)^{1/2}\bigr)^{1-2/n}.
\]
Therefore
\begin{equation}\label{eq:v-sup}
 \|v\|_\infty\le C_f(1+K)^{(n-2)/(2n)}.
\end{equation}
Because $-\Div F_g=g$, the sign in its equation is
$L_dv=g-\tau\Div f$. Subtract it from \eqref{eq:Bernoulli}:
\begin{equation}\label{eq:comparison-mu}
 L_d(\theta-v)=-\mu,\qquad
 \mu=D(u)+\tau(|f||u|-f\cdot u)\ge0.
\end{equation}
The function $D(u)$ is integrable by energy. Also
$\int|f||u|\le\|f\|_{q_*}\|u\|_{2^*}<\infty$, so $\mu\in L^1$.
Lemma~\ref{lem:initial} and \eqref{eq:v-sup} show
$\theta-v\in\Hdot\cap L^\infty$. The comparison principle
(Lemma~\ref{lem:comparison}), whose cutoff limit is taken at this fixed
solution, proves $\theta\le v$. Hence $0\le \theta_+\le v_+$ and
$\|\theta_+\|_\infty\le\|v\|_\infty$.

Finally, the energy test for $v$, before its supremum estimate is used,
also gives
\begin{equation}\label{eq:comparison-energy}
 \|\nabla v\|_2^2=\int F\cdot\nabla v,
 \qquad \|\nabla v\|_2+\|v\|_{2^*}\le C_n\|F\|_2\le C_f.
\end{equation}
Therefore
\begin{equation}\label{eq:h-2star}
 \|\theta_+\|_{2^*}\le\|v_+\|_{2^*}\le\|v\|_{2^*}\le C_f.
\end{equation}
This gain does not cost any power of $(1+K)$. The comparison function $v$ is
not assumed nonnegative; the inequality $\theta_+\le v_+$ is sufficient.
\end{proof}

\subsection{Closing the Newtonian potential}
\begin{proposition}\label{prop:closure}
Every weighted fixed point for $0\le\tau\le1$ satisfies
\begin{equation}\label{eq:closure-bound}
 K+M+\|\theta_+\|_\infty+\|\theta_+\|_{2^*}\le C_f.
\end{equation}
The estimate is valid for each fixed integer $n\ge16$, without a smallness
condition on $f$.
\end{proposition}
\begin{proof}
All quantities are finite by Lemma~\ref{lem:initial}. Since $n>6$,
Lemma~\ref{lem:potential} applies with $s=2^*$. In view of
\eqref{eq:h-2star},
\begin{equation}\label{eq:improved-potential}
 K\le C_n\|\theta_+\|_{2^*}^{4/(n-2)}\|\theta_+\|_\infty^{(n-6)/(n-2)}
 \le C_f\|\theta_+\|_\infty^{(n-6)/(n-2)}.
\end{equation}
The bound also implies $K\le C_f(1+\|\theta_+\|_\infty^{(n-6)/(n-2)})$. This only enlarges the right-hand side; the pressure remains the decaying representative in~\eqref{eq:pressure-normalization}. Indeed, replacing $p$ by $p+c$ with $c>0$ would give $(p+c+|u|^2/2)_+\ge c/2$ outside a sufficiently large ball, by~\eqref{eq:individual-decay}, and its Newtonian potential would be infinite.

Combining \eqref{eq:improved-potential} with \eqref{eq:head-bound} gives
\begin{equation}\label{eq:closure-improved}
 K\le C_f(1+K)^{(n-6)/(2n)},\qquad
 \frac{n-2}{2n}\frac{n-6}{n-2}=\frac{n-6}{2n}<1.
\end{equation}
If $K\ge1$, then
$K^{(n+6)/(2n)}\le 2^{(n-6)/(2n)}C_f$; if $K<1$, it is already bounded.
This yields a data bound for $K$. Equations \eqref{eq:Morrey-bound} and
\eqref{eq:head-bound} then bound $M$ and $\|\theta_+\|_\infty$, and \eqref{eq:h-2star} supplies the
last term of \eqref{eq:closure-bound}.

\end{proof}


\section{Completing the weighted existence argument}\label{sec:existence}
The scalar estimates must be converted into a uniform bound in $X$, not merely
a bound on compact subsets. We state the local inputs with their actual
hypotheses, prove uniform energy tightness, and retain the small rescaled
$L^2$ norm through a dimension-dependent but finite Stokes iteration.

\subsection{Local regularity and energy tightness}
\begin{lemma}\label{input:Bernoulli}
Let $n\ge2$, $r>n/2$, and let $(v,\pi)$ be a regular solution of
\[
 -\Delta v+(v\cdot\nabla)v+\nabla\pi=F,\qquad\Div v=0
 \quad\text{in }B_1,
\]
with $F\in L^\infty(B_1)$. If
\begin{equation}\label{eq:local-input}
 \|v\|_{W^{1,2}(B_1)}+\|\pi\|_{W^{1,n/(n-1)}(B_1)}
 +\|F\|_{L^\infty(B_1)}
 +\|(\pi+|v|^2/2)_+\|_{L^r(B_1)}\le C_0,
\end{equation}
then
$\|v\|_{L^\infty(B_{1/2})}+\|\nabla v\|_{L^\infty(B_{1/2})}\le C$.
The constant depends on $n,C_0$ and a positive lower bound for $r-n/2$.
\end{lemma}
This is \cite[Proposition~2.3]{LY2022}. It is a theorem about regular
solutions in a ball, and its dimension range is stated there as $n\ge2$. Its proof uses the
weighted Bernoulli estimates and the Tian--Xin vorticity criterion. We use
that established proposition as an external input.

\begin{lemma}\label{input:vorticity}
For $n\ge3$ and $M_0<\infty$, there are $\varepsilon_0>0$ and $C<\infty$
depending only on $n,M_0$ with the following property. Let $(v,\pi)$ be a
regular homogeneous stationary solution in $B_2$ with
$\|v\|_{L^2(B_2)}\le M_0$. Write
$\Omega(v)=(\partial_i v_j-\partial_j v_i)_{i,j}$. If
\begin{equation}\label{eq:vorticity-condition}
 \rho^{4-n}\int_{B_\rho(y)}|\Omega(v)|^2<\varepsilon_0
 \quad(y\in B_1,\ 0<\rho<1/2),
\end{equation}
then $\|\nabla v\|_{L^\infty(B_{1/2})}\le C$.
\end{lemma}
This is a fixed-scale consequence of \cite[Theorem~A]{LY2022}, originating
in Tian--Xin~\cite{TX1999}. To specify the reduction, take any
$x\in B_{1/2}$ and apply Theorem~A on $B_1(x)\subset B_2$. Its centers
in $B_{1/2}(x)$ lie in $B_1$, so \eqref{eq:vorticity-condition} supplies its
smallness condition with radius threshold $1/2$. The translated $L^2$ norm
is bounded by $M_0$. The theorem gives a fixed positive radius, depending
only on $n,M_0$, at which $|\nabla v(x)|$ is bounded. Take the supremum over
$x$. Only the homogeneous-force case of that theorem is required here.

\begin{corollary}\label{cor:C1}
Every weighted fixed point satisfies
\begin{equation}\label{eq:uniform-C1}
 \|u\|_\infty+\|\nabla u\|_\infty\le C_f.
\end{equation}
\end{corollary}
\begin{proof}
Apply Lemma~\ref{input:Bernoulli} with $r=n$ on every unit ball.
The energy bound gives the local $W^{1,2}$ velocity norm, since
$\|u\|_{L^2(B_1(z))}\le |B_1|^{1/n}\|u\|_{2^*}$. Because $\frac n{n-2}>\frac n{n-1}$,
\eqref{eq:basic-energy} controls the local $W^{1,\frac n{n-1}}$ pressure norm.
The force norm is bounded by $\|f\|_\infty$, and
\eqref{eq:closure-bound} bounds the $L^n$ norm of $\theta_+$ on each unit ball.
All these constants are independent of its center. The choice $r=n$ has
the strictly positive gap $r-n/2=n/2$ in every dimension under
consideration. The resulting bounds on the concentric half-balls cover
$\Rn$ and prove \eqref{eq:uniform-C1}.
\end{proof}

\begin{lemma}\label{lem:energy-tail}
For every $\varepsilon>0$ there is $R_\varepsilon$, depending only on
$\varepsilon,n,R_0,\|f\|_\infty$, such that every weighted fixed point
satisfies
\begin{equation}\label{eq:energy-tightness}
 \int_{|x|>R_\varepsilon}|\nabla u|^2<\varepsilon.
\end{equation}
\end{lemma}

\begin{proof}
Let \(i>R_0+1\), and choose \(0\le \eta_i\le 1\) such that
\[
\eta_i=0 \quad\text{on }B_i,\qquad
\eta_i=1 \quad\text{on }\mathbb R^n\setminus B_{i+1},
\qquad
|\nabla\eta_i|\le C.
\]
Thus
\[
\operatorname{supp}\nabla\eta_i
\subset A_i:=B_{i+1}\setminus B_i.
\]
For every integer \(i>R_0+1\), define
\[
e_i
:=
\int_{A_i}
\left(
|\nabla u|^2+|u|^{2^*}+|p|^{\frac n{n-2}}
\right)\,dx.
\]
Since the annuli \(A_i\) are pairwise disjoint up to null sets,
the estimates in \eqref{eq:basic-energy} imply
\[
\sum_{\substack{i\in\mathbb N\\ i>R_0+1}} e_i
\le
\|\nabla u\|_2^2
+\|u\|_{2^*}^{2^*}
+\|p\|_{\frac n{n-2}}^{\frac n{n-2}}
\le C_f.
\]
Since \(\eta_i\) is not compactly supported, we first introduce
\(\chi_R\in C_c^\infty(\mathbb R^n)\), \(R>2(i+1)\), satisfying
\[
0\le \chi_R\le 1,\qquad
\chi_R=1 \ \text{on }B_R,\qquad
\chi_R=0 \ \text{outside }B_{2R},\qquad
|\nabla\chi_R|\le \frac{C}{R}.
\]
We test \eqref{eq:homotopy} by \(u\eta_i\chi_R\). Since
\(\operatorname{supp}f\subset B_{R_0}\subset B_i\) and
\(\eta_i=0\) on \(B_i\), the force term vanishes:
\[
\tau\int_{\mathbb R^n} f\cdot u\,\eta_i\chi_R=0.
\]

For the diffusion term, integration by parts gives
\[
\begin{aligned}
\int_{\mathbb R^n}(-\Delta u)\cdot u\eta_i\chi_R
=
\int_{\mathbb R^n}\eta_i\chi_R|\nabla u|^2
+\int_{\mathbb R^n}
\chi_R\,\partial_j u_k\,u_k\,\partial_j\eta_i
+\int_{\mathbb R^n}
\eta_i\,\partial_j u_k\,u_k\,\partial_j\chi_R.
\end{aligned}
\]
Thanks to \(\operatorname{div}u=0\), we have
\[
\begin{aligned}
\int_{\mathbb R^n}(u\cdot\nabla)u\cdot u\,\eta_i\chi_R
=
-\frac12\int_{\mathbb R^n}
|u|^2\chi_R\,u\cdot\nabla\eta_i
-\frac12\int_{\mathbb R^n}
|u|^2\eta_i\,u\cdot\nabla\chi_R,
\end{aligned}
\]
and
\[
\begin{aligned}
\int_{\mathbb R^n}\nabla p\cdot u\eta_i\chi_R
=
-\int_{\mathbb R^n}
p\chi_R\,u\cdot\nabla\eta_i
-\int_{\mathbb R^n}
p\eta_i\,u\cdot\nabla\chi_R.
\end{aligned}
\]
Consequently,
\begin{align}
    \int_{\mathbb R^n}\eta_i\chi_R|\nabla u|^2
&\nonumber\le
C\int_{A_i}
\bigl(
|u||\nabla u|+|u|^3+|p||u|
\bigr)\\
&\quad
+\frac{C}{R}
\int_{B_{2R}\setminus B_R}
\bigl(
|u||\nabla u|+|u|^3+|p||u|
\bigr). \label{eq:energy-cutoff-two}
\end{align}
It remains to remove the expanding cutoff. Since \(u\) is a weighted fixed point, we have \(u\in X\) and
\[
u_i=U_{ij}*\bigl(\tau f_j-(u\cdot\nabla)u_j\bigr).
\]
Hence Lemma~\ref{lem:initial} applies. By
\eqref{eq:individual-decay}, on \(B_{2R}\setminus B_R\), we have
\[
|u|\le C_uR^{2-n},
\qquad
|\nabla u|+|p|\le C_uR^{1-n}.
\]
Then, using \(|B_{2R}\setminus B_R|\le C_nR^n\), we obtain
\[
\frac1R
\int_{B_{2R}\setminus B_R}|u||\nabla u|
\le C_uR^{2-n},
\qquad
\frac1R
\int_{B_{2R}\setminus B_R}|u|^3
\le C_uR^{5-2n},
\]
and
\[
\frac1R
\int_{B_{2R}\setminus B_R}|p||u|
\le C_uR^{2-n}.
\]
All three quantities tend to zero as \(R\to\infty\). Notice that
the constant \(C_u\) may depend on the individual fixed point; this
causes no difficulty here, since the decay is used only to justify the
limit \(R\to\infty\). Letting \(R\to\infty\) in \eqref{eq:energy-cutoff-two} and using \(\eta_i=1\) on \(\mathbb R^n\setminus B_{i+1}\), we obtain
\begin{equation}\label{eq:shell-energy}
\int_{|x|>i+1}|\nabla u|^2
\le
C\int_{A_i}
\bigl(
|u||\nabla u|+|u|^3+|p||u|
\bigr).
\end{equation}

The following estimates do not use solution-dependent constants.
Since $|A_i|\le C_ni^{n-1}$ and
$1/2^*+1/2+1/n=1$,
\begin{equation}\label{eq:shell-mixed}
 \int_{A_i}|u||\nabla u|
 \le C_ni^{(n-1)/n}
 \left(\int_{A_i}|u|^{2^*}\right)^{(n-2)/(2n)}
 \left(\int_{A_i}|\nabla u|^2\right)^{1/2}
 \le C_n(ie_i)^{(n-1)/n}.
\end{equation}
As $n\ge16$, the number $3-2^*=(n-6)/(n-2)$ is nonnegative. Thus
\[
 \int_{A_i}|u|^3
 \le \|u\|_\infty^{(n-6)/(n-2)}\int_{A_i}|u|^{2^*}
 \le C_fe_i.
\]

For the pressure term, the exponents $n/(n-2)$ and $n/2$ are
conjugate. Since $n/2\ge2^*$, interpolation on the same shell gives
\begin{equation*}
 \|u\|_{L^{n/2}(A_i)}
 \le\|u\|_\infty^{1-4/(n-2)}
       \|u\|_{L^{2^*}(A_i)}^{4/(n-2)}
 =\|u\|_\infty^{(n-6)/(n-2)}
       \left(\int_{A_i}|u|^{2^*}\right)^{2/n}.
\end{equation*}
Therefore, by \eqref{eq:uniform-C1},
\begin{equation*}
 \int_{A_i}|p||u|
 \le \|u\|_\infty^{(n-6)/(n-2)}
 \left(\int_{A_i}|p|^{n/(n-2)}\right)^{(n-2)/n}
 \left(\int_{A_i}|u|^{2^*}\right)^{2/n}\\
 \le C_fe_i.
\end{equation*}
The powers of the two integrals in the last line add to one.

For a large integer $\ell$, let
$S_\ell=\sum_{i=\ell}^{\ell+\ell^2}i^{-1}$. Then
$S_\ell\ge c\log\ell$, and
\[
 \left(\min_{\ell\le i\le\ell+\ell^2}ie_i\right)S_\ell
 \le\sum_{i=\ell}^{\ell+\ell^2}e_i\le C_f.
\]
Hence some index in this interval satisfies $ie_i\le C_f/\log\ell$.
Using that index in \eqref{eq:shell-energy}, and noting that
$i+1\le\ell+\ell^2+1$, gives
\begin{equation*}
 \int_{|x|>\ell+\ell^2+1}|\nabla u|^2
 \le C_f\bigl((\log\ell)^{-(n-1)/n}+(\log\ell)^{-1}\bigr).
\end{equation*}
This tends to zero uniformly over the fixed points and over $\tau$, proving
\eqref{eq:energy-tightness}.
\end{proof}

\subsection{A finite Stokes bootstrap for every fixed dimension}

\begin{lemma}\label{lem:interior-Stokes}
Let $n\ge3$, $1<q<\infty$, and $0<a<b$. Suppose that
$v,G\in L^q(B_b;\Rn)$ and $\pi$ is a distribution satisfying
\[
 -\Delta v+\nabla\pi=G,\qquad\Div v=0\quad\text{in }B_b.
\]
Then
\begin{equation}\label{eq:interior-Stokes}
 \|v\|_{W^{2,q}(B_a)}
 \le C_{n,q,a,b}\bigl(\|G\|_{L^q(B_b)}+\|v\|_{L^q(B_b)}\bigr).
\end{equation}
\end{lemma}

\begin{proof}
Choose $a<c<b$ and $\chi\in C_c^\infty(B_b)$ equal to one on $B_c$.
Extend $\chi G$ by zero and set
$V_i=U_{ij}*(\chi G_j)$, $\Pi=P_j*(\chi G_j)$.
For $x,y\in B_b$, the difference $x-y$ lies in $B_{2b}$. Local
integrability of $U,\nabla U$ and Young's inequality give
\[
 \|V\|_{L^q(B_b)}+\|\nabla V\|_{L^q(B_b)}
 \le C_{n,b}\|G\|_{L^q(B_b)}.
\]
The second derivatives of $U$ define order-zero Calder\'on--Zygmund
operators, with any distributional local multiples of the identity included.
Therefore
\[
 \|\nabla^2V\|_{L^q(\Rn)}\le C_{n,q}\|G\|_{L^q(B_b)}.
\]

On $B_c$, $w=v-V$ and $\rho=\pi-\Pi$ solve the homogeneous Stokes system.
Taking divergence gives $\Delta\rho=0$; applying $\Delta$ to the velocity
equation then gives $\Delta^2w=0$. Interior estimates for this
constant-coefficient elliptic equation yield
\[
 \|w\|_{W^{2,q}(B_a)}\le C_{n,q,a,c}\|w\|_{L^q(B_c)}.
\]
Combine these estimates. This is the standard potential-and-biharmonic
proof of the pressure-free local estimate, also consistent with the
Stokes estimates used in~\cite{LY2022}.
\end{proof}

\begin{proposition}\label{prop:uniform-decay}
Every weighted fixed point satisfies \eqref{eq:main-decay} with a constant
dependent only on the prescribed data. In particular,
$\|u\|_X\le C_f$.
\end{proposition}
\begin{proof}
\proofstep{I}{A fixed vorticity threshold at large distance}
Let $R\ge1$, $|x_0|=4R$, and set
\[
 v(y)=Ru(x_0+Ry),\qquad\pi(y)=R^2p(x_0+Ry).
\]
Energy and H\"older give
\begin{equation}\label{eq:rescaled-L2}
 \|v\|_{L^2(B_2)}^2
 =R^{2-n}\int_{B_{2R}(x_0)}|u|^2
 \le C_nR^{4-n}\|u\|_{2^*}^2
 \le C_fR^{4-n}.
\end{equation}
Choose $M_0$ bounding this norm for every $R\ge1$, and choose
$\varepsilon_0$ from Lemma~\ref{input:vorticity} for this fixed
$M_0$. Equation \eqref{eq:uniform-C1} supplies a uniform $r_1\in(0,1)$
such that
\begin{equation}\label{eq:small-physical-balls}
 r^{4-n}\int_{B_r(z)}|\Omega(u)|^2
 \le C_fr^4<\varepsilon_0\quad(0<r<r_1,\ z\in\Rn).
\end{equation}
By Lemma~\ref{lem:energy-tail}, choose
$R_*\ge\max\{1,R_0+1\}$, depending only on the data, so that
\begin{equation}\label{eq:far-vorticity}
 \int_{|x|>R_*}|\Omega(u)|^2<\varepsilon_0r_1^{n-4}.
\end{equation}
For $R\ge R_*$ the rescaled equation is homogeneous in $B_2$, since
$|x_0+Ry|\ge2R>R_0$ there. For $y\in B_1$ and $0<\rho<1/2$, put
$r=R\rho$, $z=x_0+Ry$. The ball $B_r(z)$ lies outside $B_{R_*}$, and
\begin{equation}\label{eq:vorticity-scaling}
 \rho^{4-n}\int_{B_\rho(y)}|\Omega(v)|^2
 =r^{4-n}\int_{B_r(z)}|\Omega(u)|^2.
\end{equation}
Use \eqref{eq:small-physical-balls} if $r<r_1$ and
\eqref{eq:far-vorticity} if $r\ge r_1$. In the latter case
$r^{4-n}\le r_1^{4-n}$, because $n>4$. Thus the hypotheses of the
vorticity criterion hold at every required center and radius. It follows
that
\begin{equation}\label{eq:rescaled-gradient}
 \|\nabla v\|_{L^\infty(B_{1/2})}\le C_f.
\end{equation}

\proofstep{II}{Retaining the small norm through every Sobolev step}
In the homogeneous equation the Stokes source is
$G=-(v\cdot\nabla)v$. Equations \eqref{eq:interior-Stokes} and
\eqref{eq:rescaled-gradient} imply
\begin{equation}\label{eq:linear-bootstrap}
 \|v\|_{W^{2,q}(B_a)}\le C_{n,q,a,b,f}\|v\|_{L^q(B_b)},
 \qquad 0<a<b\le1/2.
\end{equation}
The constant does not depend on $R$ or $x_0$. This estimate is linear in
the preceding velocity norm, although the original equation is nonlinear.

Set
\begin{equation}\label{eq:bootstrap-indices}
 m=\left\lceil\frac n4\right\rceil-1,\qquad
 a_j=\frac{2n}{n-4j}\quad(0\le j\le m),\qquad
 \rho_j=\frac12-\frac{3j}{8(m+2)}\quad(0\le j\le m+2).
\end{equation}
The sequence starts at $a_0=2$.
The integer $n-4m$ belongs to $\{1,2,3,4\}$, so the last exponent in the
sequence is at least $n/2$. For each $j<m$, the input exponent is less
than $n/2$ and
\[
 \frac1{a_{j+1}}=\frac1{a_j}-\frac2n.
\]
Apply \eqref{eq:linear-bootstrap} from $B_{\rho_j}$ to
$B_{\rho_{j+1}}$ and use Sobolev. This takes the initial $L^2$ bound
successively through the exponents in \eqref{eq:bootstrap-indices}.
More explicitly, induction using \eqref{eq:rescaled-L2} and
\eqref{eq:linear-bootstrap} gives
\begin{equation}\label{eq:bootstrap-induction}
 \|v\|_{L^{a_j}(B_{\rho_j})}
 \le C_{f,n,j}\|v\|_{L^2(B_{1/2})}
 \le C_{f,n,j}R^{(4-n)/2},\qquad 0\le j\le m.
\end{equation}
At each stage the embedding is applied on $B_{\rho_{j+1}}$ after
the Stokes estimate from $B_{\rho_j}$. Thus no additive constant
appears that could destroy the small factor $R^{(4-n)/2}$.

One more application gives $L^{2n}$ on $B_{\rho_{m+1}}$: if the input
exponent is $n/2$, use the critical embedding into the finite target
$2n$; if it is larger, use the stronger embedding. The final application
gives $W^{2,2n}$ on $B_{\rho_{m+2}}=B_{1/8}$, and hence $C^{1,1/2}$.
There are exactly $m+2$ Stokes steps, a finite number for each fixed $n$.
Every step is linear in the previous norm. Thus
\begin{equation}\label{eq:rescaled-C1-small}
 \|v\|_{C^1(B_{1/8})}\le C_fR^{(4-n)/2}.
\end{equation}
For example, when $n=19$ the chain is
\[
 L^2\longrightarrow L^{38/15}\longrightarrow L^{38/11}
 \longrightarrow L^{38/7}\longrightarrow L^{38/3}
 \longrightarrow L^{38}\longrightarrow C^1.
\]
Scaling back at $x_0$, and using \eqref{eq:uniform-C1} on the remaining
bounded region, gives
\begin{equation}\label{eq:weak-decay}
 |u(x)|\le C_f(1+|x|)^{(2-n)/2},\qquad
 |\nabla u(x)|\le C_f(1+|x|)^{-n/2}.
\end{equation}

\proofstep{III}{Two convolution improvements}
The source in \eqref{eq:fixed} now satisfies
$|G(x)|\le C_f(1+|x|)^{1-n}$. Apply \eqref{eq:convolution} with
$\beta=n-1$ and $a=2,1$ to the Stokes representation and its first
velocity derivatives. The result is
\[
 |u(x)|\le C_f(1+|x|)^{3-n},\qquad
 |\nabla u(x)|\le C_f(1+|x|)^{2-n}.
\]
The source improves to $|G(x)|\le C_f(1+|x|)^{5-2n}$.
Since $2n-5>n$, a second application, now also to $P*G$, gives
\eqref{eq:main-decay}. 
\end{proof}

\subsection{Compactness and Leray--Schauder continuation}
\begin{proof}[Proof of Theorem~\ref{thm:whole}]
\proofstep{I}{Define the map on a complete space.}
Use the real Banach space $X$ in \eqref{eq:X}. To recall completeness,
an $X$-Cauchy sequence and its first derivatives converge uniformly on
each compact set. Their limits define a $C^1$ field, the divergence
constraint passes to the limit, and taking limits in the weighted
Cauchy inequality proves convergence in $X$.
For $(\tau,v)\in[0,1]\times X$, define
\begin{equation}\label{eq:compact-map}
 T(\tau,v)=U*(\tau f-(v\cdot\nabla)v).
\end{equation}
The kernels are those of \eqref{eq:Stokes-kernels}. For
$\|v\|_X\le L$, the source obeys
\begin{equation}\label{eq:map-source-bound}
 |\tau f(x)-(v\cdot\nabla)v(x)|
 \le C_{f,L}(1+|x|)^{5-2n}.
\end{equation}
Since $2n-5>n$, \eqref{eq:convolution}, with $a=2$ and $a=1$,
gives
\begin{equation}\label{eq:map-strong-tail}
 |T(\tau,v)(x)|+(1+|x|)|\nabla T(\tau,v)(x)|
 \le C_{f,L}(1+|x|)^{2-n}.
\end{equation}
The locally integrable kernels and local Stokes estimates give a
$C^1$ output, and $\partial_iU_{ij}=0$ makes it divergence free.
Thus $T$ takes values in $X$.

\proofstep{II}{Verify joint continuity.}
For $v,w\in X$, the same convolution estimate gives the bilinear bound
\begin{equation}\label{eq:map-bilinear}
 \|U*((v\cdot\nabla)w)\|_X\le C_n\|v\|_X\|w\|_X.
\end{equation}
Indeed, the source on the left is at most
$\|v\|_X\|w\|_X(1+|x|)^{5-2n}$.
Using
$(v\cdot\nabla)v-(w\cdot\nabla)w
=((v-w)\cdot\nabla)v+(w\cdot\nabla)(v-w)$, we obtain
\begin{equation}\label{eq:map-continuity}
 \begin{split}
 \|T(\tau,v)-T(\sigma,w)\|_X
 &\le C_f|\tau-\sigma|\\
 &\quad+C_n(\|v\|_X+\|w\|_X)\|v-w\|_X.
 \end{split}
\end{equation}
This proves continuity in both variables.

\proofstep{III}{Verify compactness in the weighted norm.}
Let $(\tau_j,v_j)$ be any sequence with $\|v_j\|_X\le L$.
By \eqref{eq:map-source-bound} and Lemma~\ref{lem:interior-Stokes},
the fields $T(\tau_j,v_j)$ are bounded in $W^{2,q}(B_R)$ for every
fixed $R$ and finite $q>1$. Choose $q>n$ and use compact embedding
into $C^1$ on smaller balls. A diagonal subsequence converges in
$C^1_{\loc}$. Moreover, \eqref{eq:map-strong-tail} implies
\begin{equation}\label{eq:map-weighted-tail}
 \sup_{|x|>R}\left\{
 (1+|x|)^{n-3}|T(\tau_j,v_j)(x)|
 +(1+|x|)^{n-2}|\nabla T(\tau_j,v_j)(x)|\right\}
 \le\frac{C_{f,L}}{1+R}.
\end{equation}
Given $\varepsilon>0$, first make this common tail smaller than
$\varepsilon/4$, then use local $C^1$ convergence to make the
weighted difference on $B_R$ smaller than $\varepsilon/2$.
The subsequence is Cauchy in $X$, so completeness gives convergence
in $X$. This verifies collective compactness for all $\tau\in[0,1]$,
not merely local compactness.

\proofstep{IV}{Exclude fixed points on one common boundary.}
If $u=T(\tau,u)$, Lemma~\ref{lem:initial} first makes $(u,p)$ a
regular solution of \eqref{eq:homotopy}, with its normalized pressure
and all scalar tests admissible. Proposition~\ref{prop:closure}
gives \eqref{eq:closure-bound}, Corollary~\ref{cor:C1} gives
\eqref{eq:uniform-C1}, and Proposition~\ref{prop:uniform-decay}
then yields
\begin{equation}\label{eq:LS-apriori}
 \|u\|_X\le M_f\quad
 \text{whenever }u=T(\tau,u),\quad 0\le\tau\le1.
\end{equation}
Here $M_f$ depends only on $n,R_0,\|f\|_\infty$.
Choose $R=M_f+1$. Thus $I-T(\tau,\cdot)$ has no zero on
$\partial B_R^X$ for any parameter. These are conditional estimates
for every possible fixed point; they do not presuppose its existence.

\proofstep{V}{Compute the initial degree.}
At $\tau=0$ the map is quadratic, not identically zero.
Nevertheless, a fixed point satisfies \eqref{eq:energy-equality}
with zero right-hand side, hence $\nabla u=0$. The decay built into
$X$ forces $u=0$. This is the only fixed point of $T(0,\cdot)$.
By \eqref{eq:map-bilinear},
\begin{equation}\label{eq:LS-quadratic}
 T(0,0)=0,\qquad \|T(0,v)\|_X\le C_n\|v\|_X^2.
\end{equation}
Choose $0<\delta<R$ with $C_n\delta<1$.
If $\|v\|_X=\delta$ and $v=\lambda T(0,v)$ for
$0\le\lambda\le1$, then
$\delta\le C_n\delta^2<\delta$, a contradiction.
Normalization and homotopy invariance of degree on this small ball give
\begin{equation}\label{eq:LS-local-degree}
 \deg_{\mathrm{LS}}(I-T(0,\cdot),B_\delta^X,0)
 =\deg_{\mathrm{LS}}(I,B_\delta^X,0)=1.
\end{equation}
Uniqueness of the zero at the initial parameter permits excision from
$B_R^X$ to $B_\delta^X$. Therefore the initial degree on $B_R^X$ is
also one. Uniqueness alone would not justify that degree computation;
the quadratic estimate \eqref{eq:LS-quadratic} is what identifies it.

\proofstep{VI}{Apply continuation and recover the PDE.}
Steps I--V verify all hypotheses of the Leray--Schauder continuation
theorem, Theorem~\ref{thm:LS-continuation} in Appendix~\ref{app:LS}.
In particular,
\begin{equation}\label{eq:LS-final-degree}
 \deg_{\mathrm{LS}}(I-T(1,\cdot),B_R^X,0)
 =\deg_{\mathrm{LS}}(I-T(0,\cdot),B_R^X,0)=1.
\end{equation}
The existence property of degree produces $u=T(1,u)$.
Set $p=P_j*(f_j-u_k\partial_k u_j)$.
Lemma~\ref{lem:initial} yields \eqref{eq:NS}, regularity and
\eqref{eq:pressure-normalization}; Proposition~\ref{prop:uniform-decay}
gives the claimed decay \eqref{eq:main-decay}.
If $f$ is smooth, successive local Stokes estimates and product
estimates imply smoothness of $u,p$. No derivative of $f$ is used
in the bounds of Theorem~\ref{thm:whole}.
\end{proof}

\section{The periodic extension}\label{sec:torus}
The periodic existence theory of Frehse and R\r{u}\v{z}i\v{c}ka~\cite{FR1995a,FR1995b} is the starting point for this extension. We prove the scalar closure for every fixed $n\ge16$ and then apply periodic degree theory. The mean of the positive source and the constant mode of the resolvent are retained throughout.
In this section $\Tn=\Rn/\Z^n$ has unit volume, all integrals without a
domain are over $\Tn$, and $(-\Delta_{\T})^{-1}$ acts on mean-zero
distributions. Let $(u,p)$ first be a regular homotopy solution with
$\int u=\int p=\int f=0$. Constants $C_f$ here depend only on $n$ and
$\|f\|_\infty$. Poincar\'e, the energy identity, and periodic pressure
estimates give
\begin{equation}\label{eq:torus-energy}
 \|u\|_{W^{1,2}}+\|u\|_{2^*}+\|p\|_{\frac n{n-2}}
 +\|\nabla p\|_{\frac n{n-1}}+\|\theta_+\|_{\frac n{n-2}}\le C_f,
 \qquad \int\theta=\frac12\int|u|^2\ge0.
\end{equation}
Indeed $\|\nabla u\|_2^2=\tau\int f\cdot u$, and the zero mean of $u$
lets Poincar\'e control $\|u\|_2$ by $\|\nabla u\|_2$. Periodic Sobolev
then controls $\|u\|_{2^*}$. The source
$\tau f-u\cdot\nabla u$ is bounded in $L^{\frac n{n-1}}$ by energy. The pressure
gradient is an order-zero periodic singular integral of this source;
mean-zero Sobolev gives the pressure in $L^{\frac n{n-2}}$. 

\subsection{Periodic potential and source flux}
Let $\Gamma_0,\Psi_0$ be the mean-zero periodic kernels with
\begin{equation}\label{eq:torus-kernels}
 -\Delta\Gamma_0=\delta_0-1,\qquad -\Delta\Psi_0=\Gamma_0.
\end{equation}
Equivalently, at $k\ne0$ their Fourier coefficients are
$(4\pi^2|k|^2)^{-1}$ and $(4\pi^2|k|^2)^{-2}$, respectively, and the
coefficient at zero is zero. Choose a smooth cutoff $\chi$, supported in a
coordinate ball of radius less than $1/2$, equal to one on $B_{1/4}$, and
nonnegative. The Euclidean kernels in \eqref{eq:GammaPsi} satisfy
\begin{equation}\label{eq:kernel-smooth-remainders}
 \Gamma_0=\chi\Gamma+\gamma_0,\qquad
 \Psi_0=\chi\Psi+\psi_0,
 \qquad\gamma_0,\psi_0\in C^\infty(\Tn).
\end{equation}
For the first formula, apply $-\Delta$ to the difference; its right side
is smooth because the point singularities cancel. For the second formula,
use the first one and the identity $-\Delta\Psi=\Gamma$. Periodic
elliptic regularity proves both assertions. Consequently
\begin{equation}\label{eq:torus-signed-kernel}
 \partial_i\partial_j\Psi_0+\frac12\delta_{ij}\Gamma_0
 =\chi(x)\frac{x_ix_j}{2\omega_n|x|^n}+E_{ij}(x),
 \qquad E_{ij}\in C^\infty(\Tn).
\end{equation}
Coordinates in the singular term refer only to the ball supporting $\chi$.

Let $\delta(x)$ be torus distance to zero. Increasing a fixed constant
$c_0$ makes $\Gamma_+=\Gamma_0+c_0$ positive and gives
\begin{equation}\label{eq:positive-torus-kernel}
 c_n(1+\delta(x)^{2-n})\le\Gamma_+(x)
 \le C_n(1+\delta(x)^{2-n}).
\end{equation}
Set

\begin{equation}\label{eq:torus-KM}
 \begin{split}
 K&=\|\Gamma_+*\theta_+\|_\infty,\\
 M&=\|u\|_2^2+
 \sup_{z\in\Tn,\ 0<r\le1/8}r^{2-n}\int_{B_r(z)}|u|^2.
 \end{split}
\end{equation}
These numbers are finite for each individual regular solution.

\begin{lemma}\label{lem:torus-signed}
Every normalized regular periodic homotopy solution satisfies
\begin{equation}\label{eq:torus-Morrey}
 M\le C_f(1+K).
\end{equation}
Its periodic skew potential
\begin{equation}\label{eq:torus-d}
 d_{ij}=\partial_i(-\Delta_\T)^{-1}u_j
       -\partial_j(-\Delta_\T)^{-1}u_i
\end{equation}
is bounded for each solution and satisfies
$\partial_jd_{ij}=u_i$ and $L_d=-\Delta+u\cdot\nabla$.
\end{lemma}
\begin{proof}
The zero-mean pressure has the representation
\[
 p=\partial_i\partial_j(-\Delta_\T)^{-1}
        \left(u_i u_j-\int u_i u_j\right)
      -\tau\partial_j(-\Delta_\T)^{-1}f_j.
\]
The constant removed from $u_i u_j$ is immaterial after the derivatives,
but specifies the inverse unambiguously. The difference between this
pressure and the given one is harmonic on the torus and has zero mean,
so it vanishes. Convolving with $\Gamma_0$ and using
\eqref{eq:torus-signed-kernel} gives
\begin{equation}\label{eq:torus-signed}
 \Gamma_0*\theta(y)
 =\frac1{2\omega_n}\int\chi(x-y)
        \frac{|u(x)\cdot(x-y)|^2}{|x-y|^n}\,dx+\mathcal E(y).
\end{equation}
The smooth velocity remainder is bounded by $C_n\|u\|_2^2$. The force
term is $-\tau\partial_j\Psi_0*f_j$ and is bounded by
$C_n\|f\|_\infty$, since $\partial_j\Psi_0\in L^1(\Tn)$. Thus
$\|\mathcal E\|_\infty\le C_f$.

The positive kernel and \eqref{eq:torus-energy} give the one-sided bound
\[
 \Gamma_0*\theta
 =\Gamma_+*\theta-c_0\int\theta
 \le\Gamma_+*\theta_+\le K.
\]
In particular the nonnegative singular integral in
\eqref{eq:torus-signed} is at most $C_f(1+K)$. For $r\le1/8$, average that
bound over $y\in B_{2r}(z)$ and retain only $x\in B_r(z)$,
$y\in B_r(x)$. On this retained region $\chi(x-y)=1$, and the same
spherical calculation as in \eqref{eq:spherical-average} yields
$r^{2-n}\int_{B_r(z)}|u|^2\le C_f(1+K)$. The additional $\|u\|_2^2$ term in
\eqref{eq:torus-KM} is controlled by energy, proving
\eqref{eq:torus-Morrey}.

Regularity on the compact torus makes $u$ and its first derivatives
bounded. The local-subtraction proof used in Lemma~\ref{lem:initial},
now with the smooth periodic kernel remainders included, shows that
$d\in W^{1,\infty}$ for the individual solution. Since $\int u=0$,
\[
 \partial_jd_{ij}
 =\partial_i(-\Delta_\T)^{-1}(\Div u)
       -\Delta(-\Delta_\T)^{-1}u_i=u_i.
\]
Skew-symmetry then gives the operator identity.
\end{proof}

\begin{lemma}\label{lem:torus-flux}
Let $n\ge3$, $g\in L^{q_*}(\Tn)$, $\bar g=\int_{\Tn}g$, and put
\[
 m_{1,\T}(g)=\sup_{z\in\Tn,\ 0<r\le1/8}
                   r^{1-n}\int_{B_r(z)}|g|.
\]
If $m_{1,\T}(g)<\infty$, define
$F_g=\nabla(-\Delta_\T)^{-1}(g-\bar g)$. Then
\begin{equation}\label{eq:torus-flux-lemma}
\begin{split}
 -\Div F_g&=g-\bar g,\qquad
 \|F_g\|_{L^2(\Tn)}\le C_n\|g\|_{L^{q_*}(\Tn)},\\
 [(F_g)_{\mathrm{per}}]_{\BMO(\Rn)}
 &\le C_n\bigl(m_{1,\T}(g)+\|g\|_{L^1(\Tn)}\bigr).
\end{split}
\end{equation}
The constants concern the fixed unit torus. No coefficient or drift enters
this lemma.
\end{lemma}
\begin{proof}
Let $a$ be the mean-zero energy solution of
$-\Delta a=g-\bar g$. Testing by $a$ and using the mean-zero periodic
Sobolev inequality gives
\[
 \|\nabla a\|_2^2=\int ga
 \le\|g\|_{q_*}\|a\|_{2^*}
 \le C_n\|g\|_{q_*}\|\nabla a\|_2.
\]
This proves the $L^2$ bound for $F_g=\nabla a$ and the distributional
identity. The kernel representation is
$F_g=(\nabla\Gamma_0)*g$: the subtracted mean contributes zero because
$\int\nabla\Gamma_0=0$. In particular,
\begin{equation}\label{eq:torus-flux-L1}
 \|F_g\|_1\le\|\nabla\Gamma_0\|_1\|g\|_1\le C_n\|g\|_1.
\end{equation}

We verify the BMO bound for its Euclidean periodic extension on every
ball. Fix a small threshold $r_c=1/64$. For $r<r_c$, the ball and its
double lift to a single Euclidean chart. On $B_{2r}(z)$, the singularity
of $\nabla\Gamma_0$ has size at most $C_n|x-y|^{1-n}$; its smooth part
is bounded. The calculation in Lemma~\ref{lem:newton-flux} gives
\[
 \avint_{B_r(z)}|F_{g,\mathrm{near}}|
 \le C_nm_{1,\T}(g)+C_n\|g\|_1.
\]
For the far part, subtract its value at $z$. On shells with
$2^{j+1}r\le1/8$, the derivative bound for the periodic kernel is
$C_n(2^jr)^{-n}$, so the oscillation is bounded by
\[
 C_nr\sum_{\substack{j\ge1\\2^{j+1}r\le1/8}}
       (2^jr)^{-n}\int_{B_{2^{j+1}r}(z)}|g|
 \le C_nm_{1,\T}(g)\sum_{j\ge1}2^{-j}.
\]
After the final such shell, the distance is bounded below by a fixed
positive number. The remaining kernel derivative is bounded, so the
remaining oscillation is at most $C_nr\|g\|_1$. These estimates also
cover a possibly incomplete last shell, and prove the required small-ball
bound after replacing the comparison constant by the actual average.

For Euclidean balls of radius $r\ge r_c$, periodicity and a covering by
unit cells give
\[
 \avint_{B_r(z)}|(F_g)_{\mathrm{per}}|
 \le C_n\frac{(1+r)^n}{r^n}\|F_g\|_{L^1(\Tn)}
 \le C_n\|g\|_1.
\]
The fixed lower threshold $r_c$ is absorbed into the dimensional constant.
The mean oscillation is at most twice this average. Together with the
small-ball calculation and \eqref{eq:torus-flux-L1}, this proves
\eqref{eq:torus-flux-lemma}.
\end{proof}

Apply the lemma with
\begin{equation}\label{eq:torus-source-flux}
 g=\tau|f||u|,\qquad \bar g=\int g,\qquad
 F_g=\nabla(-\Delta_\T)^{-1}(g-\bar g),\qquad F=F_g+\tau f.
\end{equation}
Energy and the exponent identity in \eqref{eq:g-qstar} give
$\|g\|_{q_*}\le C_f$ and $0\le\bar g\le C_f$.
For $r\le1/8$, Cauchy--Schwarz and \eqref{eq:torus-Morrey} give
$\int_{B_r(z)}g\le C_fM^{1/2}r^{n-1}$. Hence
\begin{equation}\label{eq:torus-source-bounds}
 \|F\|_2\le C_f,\qquad
 [F_{\mathrm{per}}]_{\BMO}\le C_f(1+M^{1/2})\le C_f(1+K)^{1/2},
 \qquad -\Div F=g-\bar g-\tau\Div f.
\end{equation}
The mean $\bar g$ has not been discarded.

\subsection{The zero mode and its absorption}
We first give the scalar smoothing estimate used for the periodic zero mode.
It does not use an off-diagonal Gaussian estimate.

\begin{lemma}\label{lem:periodic-smoothing}
Let $b$ be a bounded skew matrix on $\Tn$. Its semigroup
$P_t=e^{-tL_b}$ is positive, contractive on $L^1$ and $L^\infty$, and
preserves constants. For $1\le q\le\infty$,
\begin{equation}\label{eq:torus-smoothing}
 \|P_t\|_{L^q\to L^\infty}
 \le C_{n,q}(1+t^{-n/(2q)}),\qquad t>0,
\end{equation}
with the usual interpretation at $q=\infty$. All constants are independent
of $\|b\|_\infty$. The positive resolvent on the full periodic space is
\begin{equation}\label{eq:resolvent}
 R_b=(1+L_b)^{-1}=\int_0^\infty e^{-t}P_t\,dt.
\end{equation}
\end{lemma}
\begin{proof}
The bounded form on $W^{1,2}(\Tn)$, with symmetric part the Dirichlet
form, defines the $L^2$ energy semigroup. One can equivalently construct
it from the implicit Euler resolvents $(I+\delta L_b)^{-1}$, whose
existence follows from the coercive form. Positive and negative
truncation tests in these resolvents give positivity and the maximum
principle. Since $L_b1=0$, the resolvents and the limiting semigroup
preserve constants. Positivity and preservation of constants give
$L^\infty$ contraction. The adjoint has skew coefficient $-b$, so it has
the same properties; duality gives $L^1$ contraction. These arguments
use no upper bound for the skew coefficient in their constants.

For $a\in L^1\cap L^2$, the solution $z(t)=P_ta$ satisfies the energy
identity
\begin{equation}\label{eq:periodic-parabolic-energy}
 \frac{d}{dt}\|z(t)\|_2^2=-2\|\nabla z(t)\|_2^2
\end{equation}
for almost every $t>0$. It follows by the usual time-regularization test
for energy solutions. Interpolating $L^2$ between $L^1$ and $L^{2^*}$
and applying periodic Sobolev yields the inhomogeneous Nash inequality
\[
 \|z\|_2^{2+4/n}
 \le C_n(\|\nabla z\|_2^2+\|z\|_2^2)\|z\|_1^{4/n}.
\]
If $m=\|a\|_1>0$, let $E(t)=\|P_ta\|_2^2/m^2$.
The $L^1$ contraction and \eqref{eq:periodic-parabolic-energy} imply
\[
 E'\le-c_nE^{1+2/n}+2E.
\]
Set $Y(t)=e^{-2t}E(t)$. Then
$Y'\le-c_ne^{4t/n}Y^{1+2/n}\le-c_nY^{1+2/n}$.
Integration, dropping the nonnegative initial reciprocal power, gives
$Y(t)\le C_nt^{-n/2}$. Thus
\[
 \|P_ta\|_2\le C_nt^{-n/4}\|a\|_1\qquad(0<t\le1).
\]

If $m=0$, the assertion is immediate. Density removes the initial
$L^2$ restriction. Apply the same argument to the adjoint and use duality
to obtain $L^2\to L^\infty$. Composition at time $t/2$ gives
$\|P_t\|_{1\to\infty}\le C_nt^{-n/2}$ for $t\le1$.
For $t\ge1$, use the estimate at time one and contraction for the remaining
time. Interpolation with the $L^\infty$ contraction proves
\eqref{eq:torus-smoothing}. Finally the semigroup Laplace transform gives
\eqref{eq:resolvent}; on $L^2$ it agrees with the inverse of the coercive
form $\int w\phi+\int(I+b)\nabla w\cdot\nabla\phi$ by uniqueness.
\end{proof}

\begin{proposition}\label{prop:torus-closure}
Every normalized regular periodic homotopy solution satisfies
\begin{equation}\label{eq:torus-closure}
 K+M+\|\theta_+\|_\infty\le C_f.
\end{equation}
Before closing, the estimates are
\begin{equation}\label{eq:torus-intermediate}
 \|\theta_+\|_\infty\le C_f(1+K)^{(n-2)/(2n)},\qquad
 K\le C_f(1+K)^{(n-4)/(2n)}.
\end{equation}
\end{proposition}
\begin{proof}
Use $R=(1+L_d)^{-1}$ on the full space, including the constants, and let
$v=R(-\Div F)$. Equations \eqref{eq:torus-source-bounds} and
\eqref{eq:scalar-sup-torus} give
\begin{equation}\label{eq:torus-v-bound}
 \|v\|_\infty\le C_f(1+K)^{(n-2)/(2n)}.
\end{equation}

The Bernoulli identity \eqref{eq:Bernoulli} holds periodically, and
$\theta=\theta_+-\theta_-$ gives
\begin{equation}\label{eq:torus-zero-mode}
 (1+L_d)\theta=\theta_++\bar g-\Div F-\widetilde\mu,
 \quad
 \widetilde\mu=D(u)+\tau(|f||u|-f\cdot u)+\theta_-\ge0.
\end{equation}
For the individual regular solution, $\theta$ is a bounded energy
function, $\widetilde\mu\in L^1$, and $R\theta_+$ belongs to
$W^{1,2}\cap L^\infty$ by the coercive formulation and the maximum
principle. Since $(1+L_d)\bar g=\bar g$, comparison applied to
$\theta-R\theta_+-\bar g-v$ yields
\begin{equation}\label{eq:torus-comparison}
 \theta\le R\theta_++\bar g+v.
\end{equation}
This explicitly accounts for the nonzero mean of the positive source $g$.

By \eqref{eq:resolvent}, contraction on $0<t<T$, and
\eqref{eq:torus-smoothing} with $q=n/(n-2)$ on later times,
\begin{align}
 \|R\theta_+\|_\infty
 &\le T\|\theta_+\|_\infty+C_n\|\theta_+\|_{\frac n{n-2}}
       \int_T^\infty e^{-t}(1+t^{-(n-2)/2})\,dt\notag\\
 &\le T\|\theta_+\|_\infty+C_{f,T}.
 \label{eq:zero-mode-absorb}
\end{align}
Fix $T=1/4$. The integral is finite; its dependence on $n$ is permitted.
Equations \eqref{eq:torus-comparison}, \eqref{eq:torus-v-bound}, and
$\bar g\le C_f$ give
$\|\theta_+\|_\infty\le \|\theta_+\|_\infty/4+C_f+C_f(1+K)^{(n-2)/(2n)}$, hence
$\|\theta_+\|_\infty\le C_f(1+K)^{(n-2)/(2n)}$.

For $0<\rho\le1/16$, the positive kernel in
\eqref{eq:positive-torus-kernel} and \eqref{eq:torus-energy} imply
\begin{equation}\label{eq:torus-potential-split}
 K\le C_f+C_n\|\theta_+\|_\infty\rho^2+C_f\rho^{4-n}.
\end{equation}
The smooth part costs $C_n\|\theta_+\|_1$; the inner singularity costs
$C_n\|\theta_+\|_\infty\rho^2$; H\"older with the exponents $n/(n-2)$ and $n/2$ controls the
outer singularity. For sufficiently large $\|\theta_+\|_\infty$, choose
$\rho=\|\theta_+\|_\infty^{-1/(n-2)}\le1/16$. For the remaining bounded range of $\|\theta_+\|_\infty$,
$K\le\|\Gamma_+\|_1\|\theta_+\|_\infty$ is already bounded. Thus
\[
 K\le C_f(1+\|\theta_+\|_\infty^{(n-4)/(n-2)})
 \le C_f(1+K)^{(n-4)/(2n)}.
\]
This is sublinear. Since $K$ was finite initially, it is uniformly
bounded, and so are $M$ and $\|\theta_+\|_\infty$. The periodic argument uses this exponent,
not the strengthened whole-space exponent from
\eqref{eq:closure-improved}.
\end{proof}

\subsection{Fixed-scale regularity and periodic degree}
\begin{proof}[Proof of Theorem~\ref{thm:extensions}(1)]
\proofstep{I}{Check the local criterion on a fixed coordinate ball.}
We first turn the scalar bounds into uniform regularity for any periodic
homotopy solution. Set $a=1/8$. Every $B_a(x_0)$ lies in a Euclidean
coordinate chart. For $y\in B_1$, define
\begin{equation}\label{eq:periodic-rescale}
 v(y)=au(x_0+ay),\quad\pi(y)=a^2p(x_0+ay),\quad
 F_a(y)=a^3\tau f(x_0+ay).
\end{equation}
These fields satisfy the same Navier--Stokes system on $B_1$. Their
relevant norms transform as
\begin{align*}
 \|v\|_{L^2(B_1)}&=a^{1-n/2}\|u\|_{L^2(B_a(x_0))},&
 \|\nabla v\|_{L^2(B_1)}&=a^{2-n/2}\|\nabla u\|_{L^2(B_a(x_0))},\\
 \|\pi\|_{L^{\frac n{n-1}}(B_1)}&=a^{2-(n-1)}\|p\|_{L^{\frac n{n-1}}(B_a(x_0))},&
 \|\nabla\pi\|_{L^{\frac n{n-1}}(B_1)}&=a^{3-(n-1)}\|\nabla p\|_{L^{\frac n{n-1}}(B_a(x_0))}.
\end{align*}
Also $(\pi+|v|^2/2)_+(y)=a^2\theta_+(x_0+ay)$, so its $L^n(B_1)$ norm is at
most $a^2|B_1|^{1/n}\|\theta_+\|_\infty$. Equation \eqref{eq:torus-energy}, the inclusion
$L^{\frac n{n-2}}(B_a)\subset L^{\frac n{n-1}}(B_a)$, and
Proposition~\ref{prop:torus-closure} control every quantity in
\eqref{eq:local-input}. All powers of $a$ are fixed dimensional
constants. Lemma~\ref{input:Bernoulli} with $r=n$ gives a
uniform $C^1$ bound on $B_{1/2}$ for $v$. Scale back and cover the torus
by finitely many $B_{a/2}(x_0)$ to obtain
\[
 \|u\|_{C^1(\Tn)}\le C_f.
\]
Since $\tau f-u\cdot\nabla u$ is bounded, periodic Stokes estimates with
zero means give all $W^{2,q}\times W^{1,q}$ bounds in
\eqref{eq:torus-main}.

\proofstep{II}{Define the compact periodic Stokes map.}
For existence, let $X_\T$ be the real Banach space of mean-zero,
divergence-free $C^1$ periodic fields. The mean-zero periodic Helmholtz
projection is defined by
\[
 \widehat{\mathbb P_\T G}(k)=
 \begin{cases}
 \left(I-\dfrac{k\otimes k}{|k|^2}\right)\widehat G(k),&k\in\Z^n\setminus\{0\},\\
 0,&k=0.
 \end{cases}
\]
Set
\[
 T_\T(\tau,v)=(-\Delta_\T)^{-1}\mathbb P_\T
                    (\tau f-(v\cdot\nabla)v).
\]
The source has mean zero: $\int f=0$ and
$(v\cdot\nabla)v=\Div(v\otimes v)$ has zero integral. Periodic Stokes
estimates make this a continuous compact map into $X_\T$ on bounded
sets: its output is bounded in $W^{2,q}$ for $q>n$, and compactly embeds
into $C^1$. Each fixed point is regular and satisfies the bounds just
proved. 
\proofstep{III}{Verify the initial degree and continue.}
At $\tau=0$, energy and the mean-zero condition force $v=0$.
The periodic Stokes estimate gives
\[
 \|T_\T(0,v)\|_{C^1}\le C_n\|v\|_{C^1}^2,
\]
which verifies \eqref{eq:LS-small-map}. The $C^1$ a priori bound
proved in Step~I is uniform in $\tau$; Step~II supplies collective
compactness and continuity. All hypotheses of
Theorem~\ref{thm:LS-continuation} hold, and it gives a fixed point
at $\tau=1$.
 The mean-zero
pressure then solves \eqref{eq:NS}, and the uniform estimates prove the
assertion. Smooth-force regularity follows by periodic elliptic
bootstrapping.
\end{proof}

\section{Support-independent estimates and noncompact forces}\label{sec:noncompact}
The constants in Theorem~\ref{thm:whole} depend on a containing ball for the
force. They cannot by themselves be used when that ball expands. We first
prove bounds depending only on
\begin{equation}\label{eq:force-N}
 N=\|f\|_\infty+\|f\|_{q_*},
\end{equation}
for compactly supported forces, and only then pass to a noncompact force.
All scalar comparisons and signed identities in the first part of this
section are applied to individual weighted fixed points, where their
admissibility was established in Lemma~\ref{lem:initial}.

\subsection{Pressure splitting and uniform force norms}
\begin{proposition}\label{prop:independent}
Let $f$ be bounded and compactly supported with
$\|f\|_\infty+\|f\|_{q_*}\le N$. Every weighted fixed point for the force
$\tau f$, $0\le\tau\le1$, satisfies \eqref{eq:noncompact-main} with a
constant $C_n(N)$ independent of the support. Its normalized pressure
splits as $p=p_N+p_f$, where $p_f=-\tau\partial_j(-\Delta)^{-1}f_j$, and
\begin{equation}\label{eq:split-pressure-bounds}
 \|p_N\|_{\frac n{n-2}}+\|\nabla p_N\|_{\frac n{n-1}}
 +\|p_f\|_2+\|\nabla p_f\|_{q_*}\le C_n(N).
\end{equation}
\end{proposition}
\begin{proof}
\proofstep{I}{Energy and pressure components}
For every $q_*\le a<\infty$, interpolation gives
\begin{equation}\label{eq:force-interpolation}
 \|f\|_a\le\|f\|_{q_*}^{q_*/a}\|f\|_\infty^{1-q_*/a}
 \le C(N).
\end{equation}
In particular $a=2$ and $a=n/2$ are permitted. The energy identity
\eqref{eq:energy-equality} gives
\begin{equation}\label{eq:independent-energy}
 \|\nabla u\|_2+\|u\|_{2^*}\le C_n(N).
\end{equation}
The order-zero operator on $u_i u_j$ gives
$\|p_N\|_{\frac n{n-2}}\le C_n\|u\|_{2^*}^2$. To estimate its gradient, use
$\partial_i(u_i u_j)=u_i\partial_i u_j$ and commute derivatives:
\[
 \partial_kp_N=\partial_k\partial_j(-\Delta)^{-1}(u_i\partial_i u_j).
\]
This is an order-zero operator applied to the convective term in
$L^{\frac n{n-1}}$. Hence
$\|\nabla p_N\|_{\frac n{n-1}}\le C_n\|u\|_{2^*}\|\nabla u\|_2$.
The order-one bound $L^{q_*}\to L^2$ gives
$\|p_f\|_2\le C_n\|f\|_{q_*}$, and order-zero estimates give
$\|\nabla p_f\|_{q_*}\le C_n\|f\|_{q_*}$. These prove
\eqref{eq:split-pressure-bounds}. 

\proofstep{II}{The signed force remainder}
The remainder from \eqref{eq:signed} has the uniform estimate
\begin{equation}\label{eq:independent-force-potential}
 \|\partial_i\Psi*f_i\|_\infty
 \le C_n(\|f\|_\infty+\|f\|_{q_*})\le C_n(N).
\end{equation}
Indeed, the kernel $|x|^{3-n}$ is integrable on $B_1$, and its restriction
to $\{|x|>1\}$ belongs to $L^{2^*}$ for $n>4$:
\[
 (n-3)2^*>n\quad\Longleftrightarrow\quad n>4.
\]
Split at radius one and use the conjugacy of $q_*,2^*$ in the far part.
The signed identity is valid for each of the present compact-force
solutions with its individual decay. Estimate
\eqref{eq:independent-force-potential}, in place of its support-radius
bound, now gives
\begin{equation}\label{eq:independent-M}
 M\le C_n(N)(1+K).
\end{equation}

\proofstep{III}{Flux estimates and the strengthened potential closure}
Equations \eqref{eq:g-qstar}--\eqref{eq:F-BMO} require only the force
norms $L^{n/2}$, $L^2$, $L^\infty$ and the velocity energy. All are
controlled independently of the support by
\eqref{eq:force-interpolation}--\eqref{eq:independent-energy}. They give
\begin{equation}\label{eq:independent-flux}
 \|F\|_2\le C_n(N),\qquad
 [F]_{\BMO}\le C_n(N)(1+K)^{1/2},\qquad
 F=\nabla (-\Delta)^{-1}(\tau|f||u|)+\tau f.
\end{equation}
The nonnegative comparison source is integrable with
\[
 \|D(u)\|_1+\tau\int(|f||u|-f\cdot u)
 \le C_n\|\nabla u\|_2^2+2\|f\|_{q_*}\|u\|_{2^*}\le C_n(N).
\]
For each fixed point the bounded skew coefficient and the bounded energy
functions required for Lemma~\ref{lem:comparison} are available. Applying
Proposition~\ref{prop:head} with the constants in
\eqref{eq:independent-flux} yields
\begin{equation}\label{eq:independent-head}
 \|\theta_+\|_\infty\le C_n(N)(1+K)^{(n-2)/(2n)},\qquad
 \|\theta_+\|_{2^*}\le C_n(N).
\end{equation}

The second inequality follows from the energy of the comparison function,
just as in \eqref{eq:comparison-energy}--\eqref{eq:h-2star}. In particular,
it does not use a full-pressure $L^{\frac n{n-2}}$ estimate. Lemma~\ref{lem:potential}
therefore gives directly
\begin{equation}\label{eq:independent-closure}
 K\le C_n(N)(1+\|\theta_+\|_\infty^{(n-6)/(n-2)})
 \le C_n(N)(1+K)^{(n-6)/(2n)}.
\end{equation}
This sublinear inequality uniformly bounds $K$, and then
\eqref{eq:independent-M} and \eqref{eq:independent-head} bound $M$ and $\|\theta_+\|_\infty$.

\proofstep{IV}{Support-independent local velocity regularity}
On each unit ball, \eqref{eq:independent-energy} controls the velocity in
$W^{1,2}$. The split pressure estimates control the full pressure in
$W^{1,\frac n{n-1}}$ on that ball: $\frac n{n-2}>\frac n{n-1}$, $2>\frac n{n-1}$, and $q_*>\frac n{n-1}$ for $n>4$.
All inclusions are on a ball of fixed volume and have constants independent
of its center. The uniform bound for $\|\theta_+\|_\infty$ supplies its $L^n$ norm there. Lemma~\ref{input:Bernoulli} with $r=n$ gives
\begin{equation}\label{eq:independent-C1}
 \|u\|_\infty+\|\nabla u\|_\infty\le C_n(N).
\end{equation}
Together these inequalities prove the proposition.
\end{proof}


\subsection{Local compactness and identification of the pressure}
\begin{proof}[Proof of Theorem~\ref{thm:extensions}(2)]
\proofstep{I}{Approximate the force with uniform bounds.}
Fix $f\in L^\infty\cap L^{q_*}$ and a function
$\eta\in C_c^\infty(B_2)$ with $0\le\eta\le1$, $\eta=1$ on $B_1$.
For positive integers $j$, put
\[
 f_j(x)=\eta(x/j)f(x).
\]
Then $\|f_j\|_\infty+\|f_j\|_{q_*}\le N$,
$f_j\to f$ strongly in $L^{q_*}$ by dominated convergence, and
$f_j=f$ on each fixed ball for sufficiently large $j$.
Theorem~\ref{thm:whole} supplies a regular solution $(u_j,p_j)$ for every
$f_j$, with Stokes decay and the pressure normalization
\begin{equation}\label{eq:approximating-pressure}
 p_j=p_{N,j}+p_{f,j}
 =\partial_i\partial_k(-\Delta)^{-1}(u_{j,i}u_{j,k})-\partial_k(-\Delta)^{-1}f_{j,k}.
\end{equation}
Proposition~\ref{prop:independent} bounds all its stated quantities
uniformly in $j$.

\proofstep{II}{Obtain local compactness of velocity and normalized pressure.}
We establish compactness of the normalized pressures before passing to a
limit. The original Stokes sources $f_j-u_j\cdot\nabla u_j$ are uniformly
bounded by \eqref{eq:independent-C1}. Apply
Lemma~\ref{lem:interior-Stokes} on balls of radii two and one to get, for
every finite $q>1$,
\begin{equation}\label{eq:uniform-local-u}
 \sup_j\sup_{x\in\Rn}\|u_j\|_{W^{2,q}(B_1(x))}\le C_{n,q}(N).
\end{equation}
The equation now gives the same bound for $\nabla p_j$ in local $L^q$.
The averages of $p_j$ are controlled by
\eqref{eq:split-pressure-bounds}, since $p_{N,j}\in L^{\frac n{n-2}}$ and
$p_{f,j}\in L^2$ with uniform norms. Poincar\'e therefore yields
\begin{equation}\label{eq:uniform-local-p}
 \sup_j\sup_{x\in\Rn}\|p_j\|_{W^{1,q}(B_1(x))}\le C_{n,q}(N).
\end{equation}
For $q>n$, these estimates give a common local H\"older modulus for
$\nabla u_j$ and $p_j$, as well as uniform bounds. Compact embedding and
a diagonal subsequence imply
\begin{equation}\label{eq:local-convergence}
 u_j\longrightarrow u\quad\text{in }C^1_{\loc},\qquad
 p_j\longrightarrow p\quad\text{in }C^0_{\loc}.
\end{equation}
The velocity is also bounded weakly in $L^{2^*}$ and its gradient weakly
in $L^2$; their weak limits are identified by the local convergence.
The equation passes to the limit in distributions, since the nonlinear
term converges locally uniformly and $f_j=f$ eventually on each compact
set. Bounds \eqref{eq:uniform-local-u}--\eqref{eq:uniform-local-p} pass
weakly on bounded sets for all finite exponents (one may first use a
countable unbounded list of exponents). Thus $(u,p)$ is regular.

\proofstep{III}{Identify the global pressure components.}
We next identify its global pressure, which is not determined by local
convergence of gradients alone. The products $u_j\otimes u_j$ are
uniformly bounded in the reflexive space $L^{\frac n{n-2}}(\Rn)$ and converge
locally. Any weak subsequential limit is therefore $u\otimes u$:
first test against compactly supported smooth functions and then use
their density in the dual space with the uniform norm bound. Boundedness
of $\partial_i\partial_k(-\Delta)^{-1}$ on $L^{\frac n{n-2}}$ gives
\begin{equation}\label{eq:pressure-component-convergence}
 p_{N,j}\rightharpoonup\partial_i\partial_k(-\Delta)^{-1}(u_i u_k)
       \quad\text{in }L^{\frac n{n-2}},\qquad
 p_{f,j}\longrightarrow-\partial_k(-\Delta)^{-1}f_k\quad\text{in }L^2.
\end{equation}
The second convergence follows from $f_j\to f$ in $L^{q_*}$ and the
order-one potential estimate. Equations
\eqref{eq:local-convergence}--\eqref{eq:pressure-component-convergence}
identify the local pressure limit with \eqref{eq:noncompact-pressure}.
There is no unidentified harmonic function or additive constant.

\proofstep{IV}{Pass the positive-part estimates to the limit.}
Put $\theta_j=p_j+|u_j|^2/2$. Then $\theta_{j,+}\to\theta_+$ locally uniformly.
Fatou's lemma gives $\|\theta_+\|_{2^*}\le C_n(N)$, while the pointwise bounds
give $\|\theta_+\|_\infty\le C_n(N)$. For each fixed $y$, positivity of the
Newtonian kernel and Fatou give
\begin{equation}\label{eq:Fatou-potential}
 (-\Delta)^{-1}\theta_+(y)\le\liminf_{j\to\infty}(-\Delta)^{-1}\theta_{j,+}(y)\le C_n(N).
\end{equation}
For the approximants the potentials have continuous representatives: use
boundedness for the local kernel and the $L^{2^*}$ bound for its far
part. Thus their uniform supremum bounds apply at every center before
Fatou is used. Alternatively the same conclusion follows from lower
semicontinuity of a positive potential and its essential bound.
Weak lower semicontinuity for the energy and velocity norms, together
with these estimates, proves \eqref{eq:noncompact-main}.

The skew potentials $d_j$ need not have uniform $L^\infty$ bounds, and
no passage to a limit in those coefficients has been made. They were
used only for individually admissible scalar comparisons before the
uniform estimates. The limit is taken in the original Navier--Stokes
system, as above.

\medskip
\proofstep{V}{Qualitative decay}
The estimates \eqref{eq:uniform-local-u}--\eqref{eq:uniform-local-p} pass
to the limit, uniformly over unit-ball centers. With $q>n$, they make
$u,\nabla u,p$ uniformly continuous on $\Rn$. A uniformly continuous
function in a finite $L^a$ space tends to zero at infinity. Indeed, if
its absolute value is at least $\varepsilon$ along a sequence escaping
to infinity, uniform continuity supplies a fixed radius on which it
is at least $\varepsilon/2$; a subsequence of these balls is disjoint,
contradicting $L^a$ integrability. Apply this to $u\in L^{2^*}$ and
$\nabla u\in L^2$.

For the pressure, use instead the finite measure of every positive
superlevel set:
\[
 \{|p|>\varepsilon\}\subset
 \{|p_N|>\varepsilon/2\}\cup\{|p_f|>\varepsilon/2\}.
\]
The two sets on the right have finite measure because
$p_N\in L^{\frac n{n-2}}$ and $p_f\in L^2$. The same disjoint-ball argument,
using uniform continuity, proves $p(x)\to0$. This proves the first
assertion of \eqref{eq:noncompact-decay-energy}; it makes no claim of
Stokes power decay for a general force in this class.

\medskip
\proofstep{VI}{The energy equality after passage to the limit}
We prove the equality afresh; weak convergence alone would only give an
inequality. Let $\zeta_R$ equal one on $B_R$, vanish outside $B_{2R}$,
satisfy $0\le\zeta_R\le1$ and $|\nabla\zeta_R|\le C/R$, and let
$A_R=B_{2R}\setminus B_R$. Local regularity justifies testing by
$u\zeta_R$. The resulting identity is
\begin{align}\label{eq:limit-energy-cutoff}
 \int\zeta_R|\nabla u|^2
 &=\int f\cdot u\,\zeta_R
   -\int u_i\partial_j u_i\,\partial_j\zeta_R
   +\frac12\int|u|^2u\cdot\nabla\zeta_R
   +\int p\,u\cdot\nabla\zeta_R.
\end{align}
Since $n\ge16$, both $3$ and $n/2$ are at least $2^*$.
The bounds $u\in L^{2^*}\cap L^\infty$ imply
$u\in L^3\cap L^{n/2}$. Hence the convection error tends to zero, and
$p_Nu\in L^1$ by the conjugacy of $\frac n{n-2}$ and $n/2$, so its pressure
error also tends to zero. For the remaining errors, the identity
$1/2^*+1/2+1/n=1$ and $|A_R|^{1/n}\le C_nR$ give
\begin{align*}
 R^{-1}\int_{A_R}|u||\nabla u|
 &\le C_n\|u\|_{L^{2^*}(A_R)}\|\nabla u\|_{L^2(A_R)}\longrightarrow0,\\
 R^{-1}\int_{A_R}|u||p_f|
 &\le C_n\|u\|_{L^{2^*}(A_R)}\|p_f\|_{L^2(A_R)}\longrightarrow0.
\end{align*}
The force term converges absolutely because $fu\in L^1$ by the
conjugate exponents $q_*,2^*$. Finally
$\int\zeta_R|\nabla u|^2\to\|\nabla u\|_2^2$ by dominated
convergence. Letting $R\to\infty$ in \eqref{eq:limit-energy-cutoff}
proves the energy equality and completes the theorem.
\end{proof}

\section{Application: far-field expansion and cancellations}\label{sec:far-field}
In this section we prove Theorem~\ref{thm:far-field}. Throughout this section, $(u,p)$ is a solution furnished by the weighted construction for a force supported in $B_{R_0}$, and $C_f$ has the dependence in Theorem~\ref{thm:whole}. Repeated spatial indices are summed unless they are explicitly fixed. The first coefficient $b$ depends only on the force, while $A$ also contains the quadratic moment of the constructed velocity. Thus the formulas give explicit coefficients in terms of $f$ and $u$, without asserting that $A$ is determined by the force independently of the choice of solution.
The link with~\cite[Proposition 4.5]{KMT2018} is explicit. When $b=0$, the effective force $G=f-\Div(u\otimes u)$ below satisfies $|G(x)|\le C_f(1+|x|)^{-(2n-3)}$, $\int G=0$, and $2n-3>n+2$, while $u(x)\to0$. That linear proposition therefore yields $u_i=-\partial_kU_{ij}A_{kj}+O(r^{-n})$ without a smallness assumption on $G$. The scalar-matrix ambiguity in $A$ also follows from~\cite[Lemma 4.1 and Remark 4.2]{KMT2018}. We give a direct calculation covering both $b=0$ and $b\ne0$, together with the gradient and pressure remainders.

\begin{proof}[Proof of Theorem~\ref{thm:far-field}]
\proofstep{I}{Moments of the effective force.}
Set $G_j=f_j-u_\ell\partial_\ell u_j=f_j-\partial_\ell(u_\ell u_j)$. The decay~\eqref{eq:main-decay} gives
\begin{equation}\label{eq:far-source-moment-bound}
 |G(y)|\le C_f(1+|y|)^{3-2n},\qquad
 \int_{\Rn}(1+|y|)^2|G(y)|\,dy+\norm{u}_2^2\le C_f.
\end{equation}
The compactly supported part of $G$ is included by increasing $C_f$. For the integrability assertions,~\eqref{eq:main-decay} gives
\[
 \begin{aligned}
 \int(1+|y|)^2|G(y)|\,dy
 &\le C_f\left(1+\int_1^\infty t^{n-1+2+3-2n}\,dt\right)
 =C_f\left(1+\frac1{n-5}\right),\\
 \norm{u}_2^2
 &\le C_f\left(1+\int_1^\infty t^{n-1+4-2n}\,dt\right)
 =C_f\left(1+\frac1{n-4}\right).
 \end{aligned}
\]

The construction and the kernels in~\eqref{eq:Stokes-kernels} give
\begin{equation}\label{eq:far-representation}
 u_i=U_{ij}*G_j,\qquad
 \partial_\ell u_i=(\partial_\ell U_{ij})*G_j,\qquad
 p=P_j*G_j.
\end{equation}
All three convolutions are absolutely convergent: the kernels have orders two, one, and one at the origin, and~\eqref{eq:far-source-moment-bound} controls their tails. The gradient identity follows first in distributions and then pointwise. This uses only the locally integrable first derivative of $U$.

To identify the moments of $G$, let $\eta_R(y)=\eta(y/R)$, where $\eta\in C_c^\infty(B_2)$ equals one on $B_1$. Integration by parts and the bound $|\nabla\eta_R|\le C/R$ give
\[
 \begin{aligned}
 \int\eta_R\partial_\ell(u_\ell u_j)
 &=-\int u_\ell u_j\partial_\ell\eta_R,\\
 \left|\int u_\ell u_j\partial_\ell\eta_R\right|
 &\le\frac C R\int_{R<|y|<2R}|u|^2
 \le\frac C R\norm{u}_2^2\longrightarrow0.
 \end{aligned}
\]
Similarly,
\[
 \begin{aligned}
 \int y_k\eta_R\partial_\ell(u_\ell u_j)
 &=-\int\eta_R u_k u_j-\int y_k u_\ell u_j\partial_\ell\eta_R,\\
 \left|\int y_k u_\ell u_j\partial_\ell\eta_R\right|
 &\le C\int_{R<|y|<2R}|u|^2\longrightarrow0.
 \end{aligned}
\]
Dominated convergence gives $\int\eta_Ru_ku_j\to\int u_ku_j$, and the moment bound~\eqref{eq:far-source-moment-bound} justifies the limits on the left. Hence
\begin{equation}\label{eq:far-effective-moments}
 \int_{\Rn}G_j=b_j,\qquad \int_{\Rn}y_kG_j=A_{kj}.
\end{equation}
In particular, the quadratic contribution to $A$ has the positive sign in~\eqref{eq:far-moments}, and $|b|+|A|\le C_f$.

\proofstep{II}{Kernel expansion and the remainder.}
Let $\mathcal K$ be a scalar component of $U$, $\nabla U$, or $P$. It is smooth off the origin and homogeneous of degree $a-n$, with $a=2$ for $U$ and $a=1$ for $\nabla U$ and $P$. Thus $|\nabla^m\mathcal K(z)|\le C_n|z|^{a-n-m}$ for $z\ne0$ and $m=0,1,2$. Fix $j$ without summing over it, and define
\[
 \mathcal R_{\mathcal K,j}(x)
 =(\mathcal K*G_j)(x)-\mathcal K(x)b_j
   +\partial_k\mathcal K(x)A_{kj}.
\]
For $r=|x|\ge2(R_0+1)$ and $|y|\le r/2$, Taylor's formula along the segment from $x$ to $x-y$ yields
\[
 \begin{aligned}
 \mathcal K(x-y)-\mathcal K(x)+y_k\partial_k\mathcal K(x)
 &=\int_0^1(1-t)y_ky_\ell\partial_k\partial_\ell\mathcal K(x-ty)\,dt,\\
 |\mathcal K(x-y)-\mathcal K(x)+y_k\partial_k\mathcal K(x)|
 &\le C_n r^{a-n-2}|y|^2.
 \end{aligned}
\]
Here $|x-ty|\ge r/2$ for $0\le t\le1$. Integration against $|G_j(y)|$ over this region contributes at most $C_n r^{a-n-2}\int|y|^2|G(y)|\,dy\le C_f r^{a-n-2}$.

Write $D_r=\{y:|y|>r/2\}$. The pointwise bound for $G$ gives
\begin{equation}\label{eq:far-tail-moments}
 \int_{D_r}|y|^m|G(y)|\,dy
 \le C_f\int_{r/2}^\infty t^{m+2-n}\,dt
 \le C_f r^{m+3-n},\qquad m=0,1.
\end{equation}
Consequently,
\[
 |\mathcal K(x)|\int_{D_r}|G|
 +|\nabla\mathcal K(x)|\int_{D_r}|y||G|
 \le C_f\bigl(r^{a-n}r^{3-n}+r^{a-n-1}r^{4-n}\bigr)
 \le C_f r^{a+3-2n}.
\]
For the remaining tail convolution, split $D_r$ according to whether $|x-y|<r/2$. On the first part, $|G(y)|\le C_f r^{3-2n}$ and local integrability of the kernel gives
\[
 \int_{D_r\cap\{|x-y|<r/2\}}|\mathcal K(x-y)|\,|G(y)|\,dy
 \le C_f r^{3-2n}\int_{|z|<r/2}|z|^{a-n}\,dz
 \le C_f r^{a+3-2n}.
\]
On the second part, $|\mathcal K(x-y)|\le C_n r^{a-n}$, so~\eqref{eq:far-tail-moments} gives the same bound. Combining the regions and using $n\ge16$ proves
\begin{equation}\label{eq:far-kernel-remainder}
 |\mathcal R_{\mathcal K,j}(x)|
 \le C_f\bigl(r^{a-n-2}+r^{a+3-2n}\bigr)
 \le C_f r^{a-n-2}.
\end{equation}
Applying this estimate separately to the three kernels in~\eqref{eq:far-representation}, and summing over $j$, yields
\begin{equation}\label{eq:far-kernel-profiles}
 \begin{aligned}
 u_i(x)&=U_{ij}(x)b_j-\partial_kU_{ij}(x)A_{kj}+R_{u,i}(x),\\
 p(x)&=P_j(x)b_j-\partial_kP_j(x)A_{kj}+R_p(x),
 \end{aligned}
\end{equation}
For the gradient remainder, apply~\eqref{eq:far-kernel-remainder} to $\partial_\ell U_{ij}$. The identities~\eqref{eq:far-representation} then give
\[
 \begin{aligned}
 R_{u,i}&=\sum_j\mathcal R_{U_{ij},j},&
 \partial_\ell R_{u,i}&=\sum_j\mathcal R_{\partial_\ell U_{ij},j},&
 R_p&=\sum_j\mathcal R_{P_j,j},\\
 |R_u|&\le C_f r^{-n},&
 |\nabla R_u|&\le C_f r^{-n-1},&
 |R_p|&\le C_f r^{-n-1}.
 \end{aligned}
\]
This proves~\eqref{eq:far-remainder}. Derivatives of the kernel of order two and three appear only away from the origin in the Taylor estimate; the convolution formula for $\nabla u$ uses only $\nabla U$.

\proofstep{III}{Compute the explicit profiles.}
For $x\ne0$, direct differentiation of~\eqref{eq:Stokes-kernels} gives
\begin{equation}\label{eq:far-kernel-derivatives}
 \begin{aligned}
 \partial_kU_{ij}(x)
 &=\frac{1}{2\omega_n}\left[
 (-\delta_{ij}x_k+\delta_{ik}x_j+\delta_{jk}x_i)r^{-n}
 -n x_i x_j x_k r^{-n-2}\right],\\
 \partial_kP_j(x)
 &=\frac{1}{\omega_n}\left[\delta_{kj}r^{-n}
 -n x_kx_jr^{-n-2}\right].
 \end{aligned}
\end{equation}
The zeroth-order contractions are $U(x)b=u^{(0)}(x)$ and $P(x)\cdot b=p^{(0)}(x)$. For the first moments, $\delta_{ij}x_kA_{kj}=(A^{\mathsf T}x)_i$, $\delta_{ik}x_jA_{kj}=(Ax)_i$, and $x_jx_kA_{kj}=x^{\mathsf T}Ax$. Therefore
\[
 \begin{aligned}
 -\partial_kU_{ij}(x)A_{kj}
 &=\frac1{2\omega_n}\left[
 ((A^{\mathsf T}-A)x)_i r^{-n}
 -(\operatorname{tr}A)x_i r^{-n}
 +n x_i(x^{\mathsf T}Ax)r^{-n-2}\right],\\
 -\partial_kP_j(x)A_{kj}
 &=\frac1{\omega_n}\left[-(\operatorname{tr}A)r^{-n}
 +n(x^{\mathsf T}Ax)r^{-n-2}\right].
 \end{aligned}
\]
Substituting $x=re$ gives~\eqref{eq:far-dipole}, hence~\eqref{eq:far-expansion}. Homogeneity and $|A|\le C_f$ give $|u^{(1)}|+r|\nabla u^{(1)}|+r|p^{(1)}|\le C_f r^{1-n}$, and~\eqref{eq:far-remainder} yields~\eqref{eq:far-leading-remainder}.

\proofstep{IV}{Prove sharpness and the cancellation criteria.}
The symmetric matrix $I+(n-2)e\otimes e$ has eigenvalues $1$ and $n-1$. It follows that
\begin{equation}\label{eq:far-leading-sharp}
 \frac{|b|r^{2-n}}{2(n-2)\omega_n}
 \le |u^{(0)}(x)|
 \le \frac{(n-1)|b|r^{2-n}}{2(n-2)\omega_n}.
\end{equation}
If $b\ne0$, put $d_n=[2(n-2)\omega_n]^{-1}$ and let $C_*$ satisfy $|u-u^{(0)}|\le C_*r^{1-n}$. Choose $R_b=\max\{2(R_0+1),2C_*/(d_n|b|)\}$. For $r\ge R_b$, $C_*r^{1-n}\le\tfrac12d_n|b|r^{2-n}$, so~\eqref{eq:far-leading-sharp} gives
\[
 \tfrac12d_n|b|r^{2-n}\le|u(x)|
 \le(n-\tfrac12)d_n|b|r^{2-n}.
\]
This proves part~(1). If $b=0$, both leading profiles vanish and~\eqref{eq:far-leading-remainder} proves~\eqref{eq:far-zero-force}. Conversely, the lower bound in part~(1) excludes $u=o(r^{2-n})$ when $b\ne0$. This proves part~(2).

For part~(3), put $W=(A-A^{\mathsf T})/2$ and $S_0=(A+A^{\mathsf T})/2-(\operatorname{tr}A)I/n$. Then $W$ is skew-symmetric, $S_0$ is symmetric and trace free, and the second velocity profile satisfies
\begin{equation}\label{eq:far-angular-cancellation}
 2\omega_n r^{n-1}u^{(1)}(x)
 =-2We+n(e^{\mathsf T}S_0e)e.
\end{equation}
Since $e\cdot We=0$, the two terms are orthogonal, and
\[
 |2\omega_n r^{n-1}u^{(1)}(re)|^2
 =4|We|^2+n^2(e^{\mathsf T}S_0e)^2.
\]
If this profile vanishes for every unit vector, then $We=0$ and $e^{\mathsf T}S_0e=0$ for every $e$. The first identity gives $W=0$. For the second, let $e_i$ be the coordinate vectors and $e_{ij}=(e_i+e_j)/\sqrt2$. Then $(S_0)_{ii}=e_i^{\mathsf T}S_0e_i=0$ and, for $i\ne j$, $2(S_0)_{ij}=2e_{ij}^{\mathsf T}S_0e_{ij}-(S_0)_{ii}-(S_0)_{jj}=0$. Thus $S_0=0$, equivalently $A=cI$. Conversely, $A=cI$ gives $A^{\mathsf T}-A=0$ and $Q_A(e)=nc-nc=0$, so both profiles in~\eqref{eq:far-dipole} vanish.

If $b=0$ and $u=O(r^{-n})$, then along each fixed direction~\eqref{eq:far-expansion} and~\eqref{eq:far-remainder} give
\[
 -2We+n(e^{\mathsf T}S_0e)e
 =\lim_{r\to\infty}2\omega_n r^{n-1}\bigl(u(re)-R_u(re)\bigr)=0.
\]
The preceding calculation yields $A=cI$. If $b=0$ and $A=cI$,~\eqref{eq:far-remainder} proves~\eqref{eq:far-scalar-moment}.

\proofstep{V}{Use the divergence-free force condition.}
Finally, suppose $\Div f=0$ in distributions. Testing against $y_j\eta_R(y)$ with $R>R_0$ gives $0=-\int_{\Rn}f_j$, because the cutoff equals one on the support of $f$. Thus $b=0$, as claimed.
\end{proof}

\appendix
\section{The Leray--Schauder fixed point  theorem}\label{app:LS}
Recall  the classical
Leray--Schauder theory; see Leray--Schauder~\cite{LS1934} and
Nirenberg~\cite{N2001}. 
\begin{theorem}\label{thm:LS-continuation}
Let $X$ be a real Banach space and let $T:[0,1]\times X\to X$ be
continuous. Assume that:
\begin{enumerate}
\item for every bounded $B\subset X$, the set
$\{T(t,x):t\in[0,1],\ x\in B\}$ is relatively compact in $X$;
\item there is $M<\infty$ such that every fixed point $x=T(t,x)$,
$0\le t\le1$, satisfies $\|x\|_X\le M$;
\item $x=0$ is the only fixed point of $T(0,\cdot)$, and
\begin{equation}\label{eq:LS-small-map}
 \lim_{\substack{x\to0\\x\ne0}}
 \frac{\|T(0,x)\|_X}{\|x\|_X}=0.
\end{equation}
\end{enumerate}
Then $T(1,\cdot)$ has a fixed point. More precisely, for every $R>M$,
\begin{equation}\label{eq:LS-degree-one}
 \deg_{\mathrm{LS}}(I-T(t,\cdot),B_R^X,0)=1
 \qquad(0\le t\le1).
\end{equation}
\end{theorem}

\section*{Declarations}
\begin{itemize}
    \item \textbf{Acknowledgments}
    W. Wang was supported by National Key R\&D Program of China (No.2023YFA1009200) and NSFC under grant 12471219.
    \item \textbf{Conflict of interest} The authors declare that they have no conflict of interest.
    \item \textbf{Data Availability} Data sharing is not applicable to this article as no datasets were generated or analyzed during the current study.
    \item \textbf{Use of artificial intelligence}
   During the preparation of this work, the authors used ChatGPT (OpenAI) in order to assist with language editing, manuscript organization, and checking mathematical arguments. After using this tool, the authors reviewed and edited the content as needed and take full responsibility for the content of the published article.
\end{itemize}

\bibliographystyle{myamsalpha}
\bibliography{Refs}

\end{document}